\documentclass{article}
\usepackage[english]{babel}
\usepackage[utf8]{inputenc}

\usepackage{bm}
\usepackage{amsmath}
\usepackage{graphicx}
\usepackage{color}
\usepackage{subcaption}
\usepackage{amsfonts}
\usepackage{float}
\usepackage{placeins}
\usepackage{algorithm}
\usepackage{algpseudocode}
\usepackage{amsmath}
\usepackage{hyperref}
\usepackage{lineno}
\RequirePackage{booktabs}
\usepackage{verbatim}
\usepackage{multirow} 
\usepackage{dsfont}

\usepackage{amstext}
\usepackage{amsthm}
\usepackage{amssymb}

\DeclareMathOperator*{\argmin}{arg\,min}

\graphicspath{ {./figs/} }

\makeatletter
\numberwithin{equation}{section}
\numberwithin{figure}{section}
\theoremstyle{plain}
\newtheorem{thm}{\protect\theoremname}
\newtheorem{lem}[thm]{\protect\lemmaname} 

\makeatother

\providecommand{\lemmaname}{Lemma}
\providecommand{\theoremname}{Theorem}

\makeatother

\title{Multicontinuum Splitting Scheme for Multiscale Flow Problems}
\author{Yalchin Efendiev\thanks{Department of Mathematics, Texas A\&M University, College Station, TX 77843, USA}, Wing Tat Leung\thanks{Department of Mathematics, City University of Hong Kong, Hong Kong}, Buzheng Shan\footnotemark[1], Min Wang\thanks{Department of Mathematics, University of Houston, Houston, TX 77004, USA }}
\date{}

\begin{document}

\maketitle

\begin{abstract}
In this paper, we propose multicontinuum splitting schemes for multiscale problems, focusing on a parabolic equation with a high-contrast coefficient.
Using the framework of multicontinuum homogenization, we introduce spatially smooth macroscopic variables and decompose the multicontinuum solution space into two components to effectively separate the dynamics at different speeds (or the effects of contrast in high-contrast cases).
By treating the component containing fast dynamics (or dependent on the contrast) implicitly and the component containing slow dynamics (or independent of the contrast) explicitly, we construct partially explicit time discretization schemes, which can reduce computational cost. The derived stability conditions are contrast-independent, provided the continua are chosen appropriately.
Additionally, we discuss possible methods to obtain an optimized decomposition of the solution space, which relaxes the stability conditions while enhancing computational efficiency. A Rayleigh quotient problem in tensor form is formulated, and simplifications are achieved under certain assumptions. 
Finally, we present numerical results for various coefficient fields and different continua to validate our proposed approach. It can be observed that the multicontinuum splitting schemes enjoy high accuracy and efficiency.


\end{abstract}

\section{Introduction}
\label{sec:introduction}

Multiscale and high-contrast properties arise in many problems, such as flows in porous media and composite materials. These problems are characterized by heterogeneities that span multiple scales and exhibit sharp contrasts in material properties, posing significant challenges in both analysis and numerical simulations. Traditional methods often struggle to efficiently capture the influence of these scales, requiring either extremely fine resolutions or approximations that compromise accuracy. For instance, in the case of a parabolic equation with a coefficient $\kappa$, the time step size for explicit methods should be $O(H^2/\max \kappa)$ to ensure stability. Considering that the mesh size $H$ needs to be fine enough to resolve the spatial region, the time step size is further significantly restricted. 
To address these issues, we propose a multicontinuum splitting scheme for multiscale problems, which greatly saves computational effort. This approach loosens the spatial mesh size by leveraging the idea of multicontinuum homogenization, and achieves a contrast-independent stability condition on the time step size by splitting the solution space.

Various methods from different perspectives have been developed to solve multiscale problems, especially time-independent ones. 
Homogenization methods \cite{papanicolau1978asymptotic,wu2002analysis,hornung2012homogenization} define upscaled coefficients through local solutions of fine-grid problems and utilize them to formulate the macroscopic equations on the coarse grid; Multiscale methods, like the Multiscale Finite Element Method (MsFEM) \cite{hou1997multiscale, hou1999convergence, efendiev2009multiscale}, the Generalized Multiscale Finite Element Method (GMsFEM) \cite{efendiev2013generalized,chung2015mixed,chung2016adaptive}, the Constraint Energy Minimizing GMsFEM (CEM-GMsFEM) \cite{chung2018constraint,chung2018constraint2,chung2023multiscale}, and the Nonlocal Multicontinua Method \cite{chung2018non}, generally construct local basis functions via cell problems to capture fine-scale details of the coefficients, and use them to form a multiscale space that approximates the solution.
For methods tailored for time-evolving equations, an effective method is the partially explicit time discretization approach \cite{chung2021contrast,chung2021contrast2}. It selects the dominant multiscale modes by using CEM-GMsFEM and handles them in an implicit fashion. Other modes are processed explicitly for efficiency. An unconditional stability with respect to the contrast can be achieved with appropriate selections of the modes. 

This study is based on the framework of multicontinuum homogenization \cite{efendiev2023multicontinuum,chung2024multicontinuum,ammosov2024multicontinuum}. It introduces macroscopic variables to represent the local averages of solution in each continuum, and posits an expansion in each Representative Volume Element (or coarse block) about the variables and multiscale basis functions. The multiscale basis functions are obtained from certain constraint local energy minimization problems solved on the fine grid to capture the multiscale property, which is inspired by \cite{chung2018constraint,chung2018non} and can promise the localization. Oversampling techniques \cite{hou1997multiscale} are utilized to mitigate boundary effects. We assume that the macroscopic variables are smooth over coarse regions, and then a system of macroscopic equations can be derived. The macroscopic system is still composed of parabolic equations if the original equation is parabolic. Besides, the fine-scale information can be recovered from macroscopic solutions by the multicontinuum expansion.

In this paper, we develop partially explicit time discretization schemes based on multicontinuum homogenization for parabolic equations with high-contrast coefficients.
We first split the continua (or the macroscopic variables) into two groups according to the speeds of the dynamics which they describe, and decompose the solution space into two components by using the multicontinuum expansion. We treat the component containing fast dynamics (typically corresponding to the continua associated with high-value regions) implicitly and the component containing slow dynamics (typically corresponding to the continua associated with low-value regions) explicitly. Then partially explicit schemes can be designed, and we show that the stability conditions only depend on the component treated explicitly. Thanks to the construction of cell problems in multicontinuum homogenization, the second component is independent of the contrast of the coefficient, provided the continua are chosen appropriately. This indicates that the decomposition indeed isolates the effects of contrast and that the stability conditions are contrast-independent. Combined with the fact that the macroscopic equations are computed on the coarse grid, such schemes significantly weaken the stability conditions compared to the explicit scheme, particularly in high-contrast cases. In comparison to the implicit scheme, the more continua included in the explicit component, the greater the computational cost savings.

In the paper, we also consider a more general case of solution space decomposition, where the decomposition is not defined a priori but is instead induced by the mixture of continua. Our goal is to find an optimal decomposition approach to separate the dynamics at different speeds such that the stability conditions on the time step are further relaxed and more continua can be included in the explicit component. Based on the estimates of the stability conditions and the properties of multicontinuum homogenization, we formulate a generalized Rayleigh quotient problem in tensor form. However, the min-max optimization problem involving tensors is costly to solve. Therefore, we introduce some assumptions and simplify it into a computationally efficient generalized eigenvalue problem. We note that our approach is applicable to general multiscale problems and is not limited to high-contrast cases.

We apply the proposed approach to various examples with different coefficient fields and numbers of continua. 
The results demonstrate that the multicontinuum splitting schemes achieve accuracy comparable to the multicontinuum implicit scheme while reducing computational effort.
Additionally, we observe that the eigenvalue problem automatically splits the continua according to the coefficient values of their corresponding regions, as expected, when the continua represent local averages of the solution in regions with different values.
This suggests that following the concept of continuum introduced in \cite{efendiev2023multicontinuum} is often the suitable choice, and this also validates the reasonableness of the concept to some degree. 
General examples, where the continua are defined in different patterns or do not represent the local averages but are defined by other auxiliary functions, are also included. The algorithm remains effective in separating modes with different speeds in these cases as well.

The paper is organized as follows. 
In Section \ref{sec:preliminaries}, we provide the preliminaries and review the main ideas of multicontinuum homogenization. 
In Section \ref{sec:schemes}, the multicontinuum splitting schemes are proposed and the stability conditions are derived.
Section \ref{sec:construction} is dedicated to the optimized construction of the space decomposition. 
We present numerical results in Section \ref{sec:numerical} and finally draw conclusions in Section \ref{sec:conclusions}.

\section{Preliminaries}
\label{sec:preliminaries}

In this paper, we consider the following parabolic partial differential equation: for a bounded polygonal domain $\Omega \subset \mathbb{R}^d\; (d=2,3)$, $u \in L^2(0,T;H^2(\Omega))\cap H^1(0,T; L^2(\Omega))$ satisfies
\begin{equation}
\label{eq:strong_form}
\partial_t u - \nabla\cdot(\kappa\nabla u) = f,
\end{equation}
where $0 < \min{\kappa} \leq \kappa(x) \leq \max{\kappa}$ is of high contrast, i.e. $\dfrac{\max{\kappa}}{\min{\kappa}} \gg 1$, and $f \in L^2(0,T;L^2(\Omega))$. The initial condition $u_0=u(\cdot,0)\in H^2 (\Omega)$ is given, and we consider zero Dirichlet boundary condition for simplicity.

The variational formulation is to find $u(\cdot,t) \in V \; (0 < t \leq T)$ such that
\begin{equation}
\label{eq:variational_form}
(\partial_{t}u,v)+a(u,v)=(f,v), \quad \forall v \in V,
\end{equation}
where $V=H_0^1(\Omega)$ and $a(u,v)=\int_{\Omega} \kappa \nabla u \cdot \nabla v$. We require that $u_0 \in V$ as well.

In this section, we briefly review the main ideas of multicontinuum homogenization method in \cite{efendiev2023multicontinuum,chung2024multicontinuum} and apply it to the weak form of the parabolic equation (\ref{eq:variational_form}). We remark that the original multicontinuum homogenization method involves the concept of Representative Volume Element (RVE) to represent the behavior of solution in the entire coarse-grid block for computational efficiency. In this paper, however, we take RVE to be the coarse block itself for illustration.

We partition the computational domain $\Omega$ into regular coarse-grid elements $K$'s, whose sizes are larger than the scale of heterogeneities. We assume that in each coarse block $K$ there are $N$ continua, which represent different average states and can often be identified by the value of high-contrast coefficient. We define the characteristic function $\psi_i$ corresponding to continuum $i$ as 
\begin{equation}
\psi_i = \begin{cases}
  1, & \text{in continuum} ~i,\\
  0, & \text{otherwise}.
\end{cases}
\end{equation} 
Then we introduce a macroscopic variable $U_i$ to represent the homogenized solution in continuum $i$, that is,  $U_i(x_K^*)=\dfrac{\int_K u \psi_i}{\int_K \psi_i}$ for some $x_K^*\in K$. Since we take $\psi_i$ to be the characteristic function, $U_i$ represents the local average of solution in continuum $i$. In general, one can define $\psi_i$'s differently by solving local spectral problems to represent each continua, and $U_i$ can have other physical meanings.

One of our most important assumptions is that macroscopic variables $U_i$ are smooth over all coarse blocks. With the smoothness, we can introduce the following general multicontinuum expansion for solution $u$ in each coarse block $K$
\begin{equation}
\label{eq:mc_full_expansion}
u = \phi_i U_i + \phi_i^m \nabla_m U_i + \phi_i^{mn} \nabla_{mn}^2 U_i + \cdots,
\end{equation}
where $\phi_i$, $\phi_i^m$, $\phi_i^{mn}$, $\ldots$ are multiscale basis functions obtained via some cell problems. The summation over repeated indices is taken. In the following, for simplicity we will use the two-term expansion 
\begin{equation}
\label{eq:mc_expansion}
u \approx \phi_i U_i + \phi_i^m \nabla_m U_i.
\end{equation}

To reduce boundary effects, we formulate cell problems in oversampled domains. We take $K^+$ to be an oversampled region constructed around $K$ and consisted of several coarse blocks, denoted by $K^p$. Here, $p$ is an index, and the target coarse block $K=K^{p_0}$ is at the center of $K^+$. We want to minimize the local energy in $K^+$ under distinct constraints by using Lagrange multipliers $\mu$. For the first cell problem, to obtain $\phi_i$'s, we impose constraints to represent the constants in the average behavior of each continuum
\begin{equation}
\label{eq:mc_cell_problem_1}
\begin{split}
&\int_{K^+}\kappa \nabla \phi_{i} \cdot \nabla v -
\sum_{j,p} \dfrac{\mu_{ij}^p}{\int_{K^p}\psi_j^p} \int_{K^p}\psi_j^p  v = 0,\\
&\int_{K^p} \phi_{i} \psi_j^p = \delta_{ij} \int_{K^p} \psi_j^p.
\end{split}
\end{equation}
To obtain $\phi_i^m$'s, we formulate the second cell problem imposing constraints to represent the linear functions on average
\begin{equation}
\label{eq:mc_cell_problem_2}
\begin{split}
&\int_{K^+}\kappa \nabla \phi^{m}_{i}\cdot \nabla v -  \sum_{j,p} \dfrac{\mu_{ij}^{mp}}{\int_{K^p}\psi_j^p} \int_{K^p}\psi_j^p v =0,\\
&\int_{K^p}  \phi^{m}_{i} \psi_j^p = \delta_{ij} \int_{K^p} (x_m-\tilde{x}_{m})\psi_j^p ,\\
&\int_{K^{p_0}} (x_m-\tilde{x}_{m})\psi_j^{p_0} =0, \\
\end{split}
\end{equation}
where $\tilde{x}_m$ is a constant. We remark that the solution in $K$ should be independent of the oversampling size when it is large enough, as a result of the construction of our cell problems. Also, the following estimates for multiscale basis functions and their gradients can be noticed
\begin{equation}
\label{eq:mc_scalings}
\begin{split}
&\|\phi_i\|=O(1),\quad \|\nabla \phi_i\|=O(\dfrac{1}{H}),\\
&\|\phi_i^{m}\|=O(H),\quad \|\nabla \phi_i^{m}\|=O(1),
\end{split}
\end{equation}
where $H$ is the size of RVE (coarse block $K$).

Now we substitute the multicontinuum expansion for solution $u$ and test function $v$
\begin{equation}
u \approx \phi_i U_i + \phi_i^m \nabla_m U_i, \quad v \approx \phi_j V_j + \phi_j^n \nabla_n V_j
\end{equation}
into the variational form (\ref{eq:variational_form}).
For the local integral of the diffusion term, we have
\begin{equation}
\begin{split}
& \int_{K} \kappa \nabla u \cdot \nabla v \\
\approx & \int_{K} \kappa \nabla (\phi_i U_i) \cdot \nabla (\phi_j V_j) 
+ \int_{K} \kappa \nabla (\phi_i U_i) \cdot \nabla (\phi_j^n \nabla_n V_j)
+ \int_{K} \kappa \nabla (\phi_i^m \nabla_m U_i) \cdot \nabla (\phi_j V_j) 
+ \int_{K} \kappa \nabla (\phi_i^m \nabla_m U_i) \cdot \nabla (\phi_j^n \nabla_n V_j) \\
\approx & \left( \int_{K} \kappa \nabla \phi_i \cdot \nabla \phi_j \right) U_i V_j
+ \left( \int_{K} \kappa \nabla \phi_i \cdot \nabla \phi_j^n \right) U_i \nabla_n V_j
+ \left( \int_{K} \kappa \nabla \phi_i^m \cdot \nabla \phi_j \right) \nabla_m U_i V_j 
+ \left( \int_{K} \kappa \nabla \phi_i^m \cdot \nabla \phi_j^n \right) \nabla_m U_i \nabla_n V_j.
\end{split}
\end{equation}
Here, we use the smoothness of $U_i$ and $V_j$, along with the assumption that the variations of both these variables and their gradients are small compared to the variations of $\phi_i$ and $\phi_i^m$. Similarly, we can approximate the term with time derivative
\begin{equation}
\begin{split}
\int_{K} \dfrac{\partial u}{\partial t} v 
\approx \int_{K} \dfrac{\partial (\phi_i U_i + \phi_i^m \nabla_m U_i)}{\partial t}(\phi_j V_j + \phi_j^n \nabla_n V_j)
\approx \int_{K} \dfrac{\partial (\phi_i U_i)}{\partial t}(\phi_j V_j)
\approx \left( \int_{K} \phi_i \phi_j \right) \dfrac{\partial}{\partial t}U_i V_j,
\end{split}
\end{equation}
where for the second approximate equality we use the estimates for $\phi_i$ and $\phi_i^m$ as stated in (\ref{eq:mc_scalings}).

By summing all the local integrals over $K$, we obtain
\begin{equation}
\label{eq:mc_derivation_step}
\begin{split}
& \int_{\Omega} \dfrac{\partial u}{\partial t} v + \int_{\Omega} \kappa \nabla u \cdot \nabla v \\
\approx & \int_{\Omega} \gamma_{ji} \dfrac{\partial}{\partial t} U_i V_j + \int_{\Omega} \alpha_{ji} U_i V_j + \int_{\Omega} \alpha_{ji}^{*n} U_i \nabla_n V_j + \int_{\Omega} \alpha_{ji}^{m*} \nabla_m U_i V_j + \int_{\Omega} \alpha_{ji}^{mn} \nabla_m U_i \nabla_n V_j,
\end{split}
\end{equation}
and
\begin{equation}
\begin{split}
\int_{\Omega} f v 
\approx \int_{\Omega} f_j V_j,
\end{split}
\end{equation}
by defining the effective properties as
\begin{equation}
\begin{split}
& \alpha_{ij} = \dfrac{1}{|K|} \int_{K} \kappa \nabla \phi_j \cdot \nabla \phi_i, \quad
\alpha_{ij}^{m*} = \dfrac{1}{|K|} \int_{K} \kappa \nabla \phi_j^m \cdot \nabla \phi_i, \quad
\alpha_{ij}^{*m} = \dfrac{1}{|K|} \int_{K} \kappa \nabla \phi_j \cdot \nabla \phi_i^m, \\
& \alpha_{ij}^{mn} = \dfrac{1}{|K|} \int_{K} \kappa \nabla \phi_j^m \cdot \nabla \phi_i^n, \quad
\gamma_{ij} = \dfrac{1}{|K|} \int_{K} \phi_j \phi_i, \quad
f_j = \dfrac{1}{|K|} \int_{K} f \phi_j.
\end{split}
\end{equation}
Since the summation of the third and fourth terms in (\ref{eq:mc_derivation_step}) is negligible by integration by parts, by the smoothness of $U_i$ and $V_j$, we have the following system of macroscopic parabolic equations
\begin{equation}
\gamma_{ji} \dfrac{\partial}{\partial t} U_i + \alpha_{ji} U_i - \nabla_n (\alpha_{ji}^{mn} \nabla_m U_i) = f_j,
\end{equation}
for any $j=1,2,\ldots, N$.
The macroscopic equations will be solved on the coarse grid to obtain the homogenized solutions.

\section{Multicontinuum splitting schemes}
\label{sec:schemes}

In this section, we introduce the splitting schemes based on multicontinuum homogenization and derive the stability results.

We denote the regular coarse mesh as $\mathcal{T}_H$, the coarse block with center $x_l$ as $K_{H}(x_l)$, the collection of their centers as $I_H$, and the number of continua as $N$.
Following the idea of multicontinuum homogenization, we introduce the multicontinuum space
\begin{equation}
\begin{split}
V_{mc}=\left\{ \sum_{x_{l} \in I_H} \mathds{1}_{K_{H}(x_l)} \sum_{i=1}^{N}  \left( \phi_{i}^{K_{H}(x_l)} U_{i} + \tilde{\phi}_{i}^{K_{H}(x_l)}\cdot\nabla U_{i} \right): \; U_i \in \mathcal{F}(\Omega) \right\} = \oplus_{i=1}^N \tilde{V_{i}},
\end{split}
\end{equation}
where $\phi_{i}^{K_{H}(x_l)}$ represents the constant in the average behavior of continuum $i$ for each coarse block $K_{H}(x_l)$ as defined in (\ref{eq:mc_cell_problem_1}), $\tilde{\phi}_{i}^{K_{H}(x_l)} = (\phi_{i}^{K_{H}(x_l),m})_m$ represents the linear functions as defined in (\ref{eq:mc_cell_problem_2}), and $\mathcal{F}(\Omega)$ is a function space defined on $\Omega$ with certain smoothness. Later for convenience, we may omit the superscript $K_{H}(x_l)$ when the context is clear. We remark that one of the main assumptions of multicontinuum homogenization is that multicontinuum variables $U_i$ are smooth over all coarse blocks. Under some conditions as specified in \cite{leung2024some}, we can choose $\mathcal{F}=H_{0}^{1+\alpha}(\Omega)$ for some $\alpha \in (0,1]$, which can describe the smoothness and guarantee the convergence of the method. Besides, the $x_l$ in the definition is discrete; for the continuous version, refer to \cite{leung2024some}.

To distinguish the effects of constants and linear functions from each continuum $i$, we can define $\tilde{T}_{i} = \tilde{T}_{i,0}+\tilde{T}_{i,1}: \; \mathcal{F} \rightarrow \tilde{V}_{i}$ by
\begin{equation}
\begin{split}
\tilde{T}_{i,0}(U_i) = \sum_{x_{l} \in I_H} \mathds{1}_{K_{H}(x_l)} \phi_{i}^{K_{H}(x_l)} U_{i},\\
\tilde{T}_{i,1}(U_i) = \sum_{x_{l} \in I_H} \mathds{1}_{K_{H}(x_l)} \tilde{\phi}_{i}^{K_{H}(x_l)} \cdot \nabla U_{i},
\end{split}
\end{equation}
and therefore $\tilde{V}_{i} = \tilde{T}_{i} (\mathcal{F})$.

We consider the following variational form: find $u(t,\cdot) \in V_{mc}$ such that 
\begin{equation}
\label{eq:mc_variational}
(\partial_{t}u,v)+a(u,v)=(f,v), \quad \forall v \in V_{mc}.
\end{equation}
In order to get a partially explicit time discretization scheme, we want to decompose $V_{mc}$ into two subspaces and treat one of them explicitly.
Let $I_1$ and $I_2$ be two disjoint sets whose union is the index set for continua, \{1,2,\ldots,N\}, and we have $V_{mc}=V_1+V_2$, where
\begin{equation}
\label{eq:first_decomposition}
\begin{split}
V_j = \left\{ \sum_{x_{l} \in I_H} \mathds{1}_{K_{H}(x_l)} \sum_{i\in I_{j}} \left( \phi_{i}^{K_{H}(x_l)} U_{i} + \tilde{\phi}_{i}^{K_{H}(x_l)}\cdot\nabla U_{i} \right): \; U_i \in \mathcal{F}(\Omega) \right\}.
\end{split}
\end{equation}
Then (\ref{eq:mc_variational}) is equivalent to considering $u_1 \in V_1$ and $u_2 \in V_2$ satisfying
\begin{equation}
\label{eq:mc_variational_split_1}
\begin{split}
(\partial_{t}(u_{1}+u_{2}),v_{1})+a(u_{1},v_{1})+a(u_{2},v_{1}) & =(f,v_{1}), \quad \forall v_1\in V_{1},\\
(\partial_{t}(u_{1}+u_{2}),v_{2})+a(u_{1},v_{2})+a(u_{2},v_{2}) & =(f,v_{2}), \quad \forall v_2\in V_{2}.
\end{split}
\end{equation}
Moreover, we define $T_{j} = T_{j,0} + T_{j,1}: \; \prod_{i\in I_j} \mathcal{F} \rightarrow V_j$, for $j=1,2$, by
\begin{equation}
T_{j,k}(U)= T_{j,k}((U_i)_{i\in I_{j}}) = \sum_{i\in I_{j}}\tilde{T}_{i,k}(U_{i}),
\end{equation}
and define the bilinear forms $m_{ij}(\cdot,\cdot)$, $a_{ij}(\cdot,\cdot)$, $b_{ij}(\cdot,\cdot)$ and $c_{ij}(\cdot,\cdot)$ by
\begin{equation}
\begin{split}
m_{ij}(U,W) =(T_{j}U,T_{i}W), \quad a_{ij}(U,W) =a(T_{j,1}U,T_{i,1}W), \\
b_{ij}(U,W) =a(T_{j,1}U,T_{i,0}W), \quad c_{ij}(U,W) =a(T_{j,0}U,T_{i,0}W),
\end{split}
\end{equation}
for any $U \in \prod_{k\in I_j} \mathcal{F}$ and $W \in \prod_{k\in I_i} \mathcal{F}$.
With these notations, we can rewrite (\ref{eq:mc_variational_split_1}) as finding $U_1 \in \prod_{k\in I_1} \mathcal{F}$ and $U_2 \in \prod_{k\in I_2} \mathcal{F}$ such that
\begin{equation}
\begin{split}
(\partial_{t}(T_{1}U_{1}+T_{2}U_{2}),T_{1}W_{1})+a(T_{1}U_{1},T_{1}W_{1})+a(T_{2}U_{2},T_{1}W_{1}) & =(f,T_{1}W_{1}),\quad \forall T_{1}W_{1} \in V_{1},\\
(\partial_{t}(T_{1}U_{1}+T_{2}U_{2}),T_{2}W_{2})+a(T_{1}U_{1},T_{2}W_{2})+a(T_{2}U_{2},T_{2}W_{2}) & =(f,T_{2}W_{2}),\quad \forall T_{2}W_{2} \in V_{2},
\end{split}
\end{equation}
or for $i=1,2$,
\begin{equation}
\begin{split}
\sum_{j=1}^{2}m_{ij}(\partial_{t}U_{j},W_{i})+\sum_{j=1}^{2}a_{ij}(U_{j},W_{i})\\
+\sum_{j=1}^{2}b_{ij}(U_{j},W_{i})+\sum_{j=1}^{2}b_{ji}(W_{i},U_{j})+\sum_{j=1}^{2}c_{ij}(U_{j},W_{i}) & =(f,T_{i}W_{i}),\quad\forall W_{i}\in \prod_{k\in I_i} \mathcal{F}.
\end{split}
\end{equation}
To simplify it, we note that $b_{ij}(U_{j},W_{i})+b_{ji}(W_{i},U_{j})=0$ by integration by parts \cite{efendiev2023multicontinuum}, and we write $(f,T_{i}W_{i})=(f_{i},W_{i})$.

Now we can select finite dimensional space $V_{i,H} \subset \prod_{k\in I_i} \mathcal{F}$ $(i=1,2)$, like Lagrange finite element space, and design the following partially explicit time discretization schemes:

\textbf{Discretization scheme 1:} finding $U_{1}^{n+1}\in V_{1,H}$ and $U_{2}^{n+1}\in V_{2,H}$ such that
\begin{equation}
\label{eq:Scheme_1}
\begin{split}
m_{11}(\cfrac{U_{1}^{n+1}-U_{1}^{n}}{\tau},W_{1})+ m_{12}(\cfrac{U_{2}^{n}-U_{2}^{n-1}}{\tau},W_{1}) +a_{11}(U_{1}^{n+1},W_{1})+a_{12}(U_{2}^{n},W_{1})\\
+c_{11}(U_{1}^{n+1},W_{1})+c_{12}(U_{2}^{n+1},W_{1}) & =(f_{1},W_{1}),\quad \forall W_{1}\in V_{1,H},\\
m_{22}(\cfrac{U_{2}^{n+1}-U_{2}^{n}}{\tau},W_{2})+m_{21}(\cfrac{U_{1}^{n}-U_{1}^{n-1}}{\tau},W_{2})+a_{21}(U_{1}^{n+1},W_{2})+a_{22}(U_{2}^{n},W_{2})\\
+c_{21}(U_{1}^{n+1},W_{2})+c_{22}(U_{2}^{n+1},W_{2}) & =(f_{2},W_{2}),\quad \forall W_{2}\in V_{2,H}.
\end{split}
\end{equation}

\textbf{Discretization scheme 2:} finding $U_{1}^{n+1}\in V_{1,H}$ and $U_{2}^{n+1}\in V_{2,H}$ such that
\begin{equation}
\label{eq:Scheme_2}
\begin{split}
m_{11}(\cfrac{U_{1}^{n+1}-U_{1}^{n}}{\tau},W_{1})+ m_{12}(\cfrac{U_{2}^{n}-U_{2}^{n-1}}{\tau},W_{1}) +a_{11}(U_{1}^{n+1},W_{1})+a_{12}(U_{2}^{n},W_{1})\\
+c_{11}(U_{1}^{n+1},W_{1})+c_{12}(U_{2}^{n},W_{1}) & =(f_{1},W_{1}),\quad \forall W_{1}\in V_{1,H},\\
m_{22}(\cfrac{U_{2}^{n+1}-U_{2}^{n}}{\tau},W_{2})+m_{21}(\cfrac{U_{1}^{n}-U_{1}^{n-1}}{\tau},W_{2})+a_{21}(U_{1}^{n+1},W_{2})+a_{22}(U_{2}^{n},W_{2})\\
+c_{21}(U_{1}^{n+1},W_{2})+c_{22}(U_{2}^{n},W_{2}) & =(f_{2},W_{2}),\quad \forall W_{2}\in V_{2,H}.
\end{split}
\end{equation}
The initial conditions are projected into the corresponding finite element spaces.

The idea of these discretization schemes is to treat the component of solution in $V_{1,H}$, or say $T_1(V_{1,H})$, implicitly, and treat the component in $V_{2,H}$, or say $T_2(V_{2,H})$, explicitly to enhance computational efficiency. Intuitively, it is desirable for $T_1(V_{1,H})$ to include the fast dynamics (contrast-dependent modes) and for $T_2(V_{2,H})$ to include the slow dynamics (contrast-independent modes). We expect that the stability conditions for the partially explicit schemes are only dependent on the explicit component. 
The two schemes differ in how the $c(\cdot,\cdot)$ term is handled. In discretization scheme 1, we use implicit scheme for both $V_{1, H}$ and $V_{2, H}$ in the $c(\cdot,\cdot)$ term; while in discretization scheme 2, we use implicit scheme for $V_{1, H}$ and explicit scheme for $V_{2, H}$ in the $c(\cdot,\cdot)$ term. We remark that the $c(\cdot,\cdot)$ corresponds to the reaction term in the partial differential equation, and that term can be large compared to the mesh size if the continua are defined in regions with sizes much smaller than the mesh size. In this case, discretization scheme 1 can give us a more stable scheme, and since $c(\cdot,\cdot)$ corresponds to the reaction term, it makes the corresponding matrix a mass matrix-like matrix with a conditional number independent of $H$.

To prove the stability results, we need to utilize the strengthened Cauchy-Schwarz inequality \cite{aldaz2009strengthened}.

\begin{lem}[Strengthened Cauchy-Schwarz inequality]
Let $H$ be a Hilbert space and $F_1,F_2 \subset H$ be two finite dimensional subspaces. If $F_1 \cap F_2 = \{0\}$, then there exists a constant $\gamma=\gamma(F_1,F_2)\in [0,1)$ such that 
\[|(x,y)| \leq \gamma \|x\| \|y\|,\]
for all $x \in F_1$ and $y \in F_2$.
\end{lem}

Since $V_1 \cap V_2 = \{0\}$, and $T_1(V_{1,H}) \subset V_1 \subset L^2$ and $T_2(V_{2,H}) \subset V_2 \subset L^2$ are of finite dimensions, we can define a constant $\gamma \in (0,1)$, depending on $T_1(V_{1,H})$ and $T_2(V_{2,H})$, by
\begin{equation}
\gamma = \sup_{U_1 \in V_{1,H}, U_2 \in V_{2,H}} \dfrac{(T_1 U_1, T_2 U_2)}{\|T_1 U_1\| \|T_2 U_2\|},
\end{equation} 
and therefore
\begin{equation}
\begin{split}
m_{12}(U_2,U_1) = m_{21}(U_1,U_2) =  (T_1 U_1, T_2 U_2) \leq \gamma \|T_1 U_1\| \|T_2 U_2\| = \gamma \|U_1\|_{m_{11}} \|U_2\|_{m_{22}}.
\end{split}
\end{equation}
We consider the case when $f_1=f_2=0$ to facilitate the analysis.

\begin{thm}
For the discretization scheme 1 in (\ref{eq:Scheme_1}), we have the following stability result
\begin{equation}
\begin{split}
    &\dfrac{\gamma^2}{\tau} \sum_{i=1}^{2} \| U_{i}^{n+1}-U_{i}^{n} \|_{m_{ii}}^2 + \| (U_{1}^{n+1},U_{2}^{n+1}) \|_{a}^2 + \| (U_{1}^{n+1},U_{2}^{n+1}) \|_{c}^2 \\
    \leq
    &\dfrac{\gamma^2}{\tau} \sum_{i=1}^{2} \| U_{i}^{n}-U_{i}^{n-1} \|_{m_{ii}}^2 + \| (U_{1}^{n},U_{2}^{n}) \|_{a}^2 + \| (U_{1}^{n},U_{2}^{n}) \|_{c}^2 ,
\end{split}
\end{equation}
if the stability condition
\begin{equation}
\label{eq:Scheme_1_stability}
    \tau \leq (1-\gamma^2) \inf_{W\in V_{2,H}}\cfrac{\|W\|_{m_{22}}^{2}}{\|W\|_{a_{22}}^{2}}
\end{equation}
is satisfied, where
\begin{equation}
\begin{split}
    \|(U_{1},U_{2})\|_{a}^{2}:=\sum_{ij} a_{ij} (U_{j},U_{i})=a(\sum_{i}T_{i,1}U_{i},\sum_{i}T_{i,1}U_{i}),\\
    \|(U_{1},U_{2})\|_{c}^{2}:=\sum_{ij} c_{ij} (U_{j},U_{i})=a(\sum_{i}T_{i,0}U_{i},\sum_{i}T_{i,0}U_{i}).
\end{split}
\end{equation}
\end{thm}

\begin{proof}

We consider $W_1 = U_{1}^{n+1}-U_{1}^{n}$ and $W_2 = U_{2}^{n+1}-U_{2}^{n}$ in (\ref{eq:Scheme_1}) and obtain
\begin{equation}
\label{eq:Thm_1_main_eqs}
\begin{split}
\dfrac{1}{\tau} m_{11}(U_{1}^{n+1}-U_{1}^{n},U_{1}^{n+1}-U_{1}^{n})+ \dfrac{1}{\tau} m_{12}(U_{2}^{n}-U_{2}^{n-1},U_{1}^{n+1}-U_{1}^{n}) + a_{11}(U_{1}^{n+1},U_{1}^{n+1}-U_{1}^{n})+a_{12}(U_{2}^{n},U_{1}^{n+1}-U_{1}^{n})\\
+c_{11}(U_{1}^{n+1},U_{1}^{n+1}-U_{1}^{n})+c_{12}(U_{2}^{n+1},U_{1}^{n+1}-U_{1}^{n}) & = 0,\\
\dfrac{1}{\tau}m_{22}(U_{2}^{n+1}-U_{2}^{n},U_{2}^{n+1}-U_{2}^{n}) + \dfrac{1}{\tau} m_{21}(U_{1}^{n}-U_{1}^{n-1},U_{2}^{n+1}-U_{2}^{n}) + a_{21}(U_{1}^{n+1},U_{2}^{n+1}-U_{2}^{n})+a_{22}(U_{2}^{n},U_{2}^{n+1}-U_{2}^{n})\\
+c_{21}(U_{1}^{n+1},U_{2}^{n+1}-U_{2}^{n})+c_{22}(U_{2}^{n+1},U_{2}^{n+1}-U_{2}^{n}) & = 0.
\end{split}
\end{equation}

Adding up the first two terms in the above two equations and using the strengthened Cauchy-Schwarz inequality gives us
\begin{equation}
\begin{split}
&\dfrac{1}{\tau} \|U_{1}^{n+1}-U_{1}^{n}\|_{m_{11}}^2 + \dfrac{1}{\tau} m_{12}(U_{2}^{n}-U_{2}^{n-1},U_{1}^{n+1}-U_{1}^{n})
+ \dfrac{1}{\tau} \|U_{2}^{n+1}-U_{2}^{n}\|_{m_{22}}^2 + \dfrac{1}{\tau} m_{21}(U_{1}^{n}-U_{1}^{n-1},U_{2}^{n+1}-U_{2}^{n}) \\
\geq
&\dfrac{1}{\tau} \left( 
\|U_{1}^{n+1}-U_{1}^{n}\|_{m_{11}}^2 - \gamma \| U_{1}^{n+1}-U_{1}^{n} \|_{m_{11}} \| U_{2}^{n}-U_{2}^{n-1} \|_{m_{22}} 
+ \|U_{2}^{n+1}-U_{2}^{n}\|_{m_{22}}^2 - \gamma \| U_{1}^{n}-U_{1}^{n-1} \|_{m_{11}} \| U_{2}^{n+1}-U_{2}^{n} \|_{m_{22}}
\right) \\
\geq
&\dfrac{1}{2\tau} \left( 
\|U_{1}^{n+1}-U_{1}^{n}\|_{m_{11}}^2 - \gamma^2 \| U_{2}^{n}-U_{2}^{n-1} \|_{m_{22}}^2
+ \|U_{2}^{n+1}-U_{2}^{n}\|_{m_{22}}^2 - \gamma^2 \| U_{1}^{n}-U_{1}^{n-1} \|_{m_{11}}^2
\right).
\end{split}
\end{equation}

To estimate the summation over the terms about $a_{ij}$, we have
\begin{equation}
\label{eq:Thm_1_sum_a_ij}
\begin{split}
&a_{11}(U_{1}^{n+1},U_{1}^{n+1}-U_{1}^{n}) + a_{12}(U_{2}^{n},U_{1}^{n+1}-U_{1}^{n}) + a_{21}(U_{1}^{n+1},U_{2}^{n+1}-U_{2}^{n})+a_{22}(U_{2}^{n},U_{2}^{n+1}-U_{2}^{n}) \\
= &a( T_{1,1}U_{1}^{n+1} + T_{2,1}U_{2}^{n}, \; T_{1,1}(U_{1}^{n+1}-U_{1}^{n}) + T_{2,1}(U_{2}^{n+1}-U_{2}^{n}) ) \\
= &a( T_{1,1}U_{1}^{n+1} + T_{2,1}U_{2}^{n+1}, \; T_{1,1}(U_{1}^{n+1}-U_{1}^{n}) + T_{2,1}(U_{2}^{n+1}-U_{2}^{n}) ) \\ 
& - a( T_{2,1}U_{2}^{n+1} - T_{2,1}U_{2}^{n}, \; T_{1,1}(U_{1}^{n+1}-U_{1}^{n}) + T_{2,1}(U_{2}^{n+1}-U_{2}^{n}) ).
\end{split}
\end{equation}
It can be noticed that for the first term, 
\begin{equation}
\begin{split}
&a( T_{1,1}U_{1}^{n+1} + T_{2,1}U_{2}^{n+1}, \; T_{1,1}(U_{1}^{n+1}-U_{1}^{n}) + T_{2,1}(U_{2}^{n+1}-U_{2}^{n}) ) \\
= &a(T_{1,1}U_{1}^{n+1} + T_{2,1}U_{2}^{n+1}, \; T_{1,1}U_{1}^{n+1} + T_{2,1}U_{2}^{n+1}) - a(T_{1,1}U_{1}^{n+1} + T_{2,1}U_{2}^{n+1}, \; T_{1,1}U_{1}^{n} + T_{2,1}U_{2}^{n}) \\
= &\| (U_{1}^{n+1},U_{2}^{n+1}) \|_{a}^2 - \dfrac{1}{2} \left( \| (U_{1}^{n+1},U_{2}^{n+1}) \|_{a}^2 + \| (U_{1}^{n},U_{2}^{n}) \|_{a}^2 - \| (U_{1}^{n+1}-U_{1}^{n} ,U_{2}^{n+1}-U_{2}^{n}) \|_{a}^2 \right) \\
= & \dfrac{1}{2} \left( \| (U_{1}^{n+1},U_{2}^{n+1}) \|_{a}^2 +  \| (U_{1}^{n+1}-U_{1}^{n} ,U_{2}^{n+1}-U_{2}^{n}) \|_{a}^2 -  \| (U_{1}^{n},U_{2}^{n}) \|_{a}^2 \right),
\end{split}
\end{equation}
and for the second term,
\begin{equation}
\begin{split}
& a( T_{2,1}U_{2}^{n+1} - T_{2,1}U_{2}^{n}, \; T_{1,1}(U_{1}^{n+1}-U_{1}^{n}) + T_{2,1}(U_{2}^{n+1}-U_{2}^{n}) ) \\
\leq & \frac{1}{2} \left(  \| U_{2}^{n+1} - U_{2}^{n} \|_{a_{22}}^2 + \| (U_{1}^{n+1}-U_{1}^{n} ,U_{2}^{n+1}-U_{2}^{n}) \|_{a}^2 \right).
\end{split}
\end{equation}
Substituting the estimates back into (\ref{eq:Thm_1_sum_a_ij}), one can obtain
\begin{equation}
\begin{split}
&a_{11}(U_{1}^{n+1},U_{1}^{n+1}-U_{1}^{n}) + a_{12}(U_{2}^{n},U_{1}^{n+1}-U_{1}^{n}) + a_{21}(U_{1}^{n+1},U_{2}^{n+1}-U_{2}^{n})+a_{22}(U_{2}^{n},U_{2}^{n+1}-U_{2}^{n}) \\
\geq & \dfrac{1}{2} \left( \| (U_{1}^{n+1},U_{2}^{n+1}) \|_{a}^2 -  \| (U_{1}^{n},U_{2}^{n}) \|_{a}^2 - \| U_{2}^{n+1} - U_{2}^{n} \|_{a_{22}}^2 \right).
\end{split}
\end{equation}

Similarly, we can also get a lower bound for the summation over $c_{ij}$
\begin{equation}
\begin{split}
&c_{11}(U_{1}^{n+1},U_{1}^{n+1}-U_{1}^{n})+c_{12}(U_{2}^{n+1},U_{1}^{n+1}-U_{1}^{n}) + c_{21}(U_{1}^{n+1},U_{2}^{n+1}-U_{2}^{n})+c_{22}(U_{2}^{n+1},U_{2}^{n+1}-U_{2}^{n}) \\
= &a( T_{1,0}U_{1}^{n+1} + T_{2,0}U_{2}^{n+1}, \; T_{1,0}(U_{1}^{n+1}-U_{1}^{n}) + T_{2,0}(U_{2}^{n+1}-U_{2}^{n}) ) \\
= &a( T_{1,0}U_{1}^{n+1} + T_{2,0}U_{2}^{n+1}, \; T_{1,0}U_{1}^{n+1} + T_{2,0}U_{2}^{n+1} ) - a( T_{1,0}U_{1}^{n+1} + T_{2,0}U_{2}^{n+1}, \; T_{1,0}U_{1}^{n} + T_{2,0}U_{2}^{n} ) \\
\geq &\dfrac{1}{2} \left( \| (U_{1}^{n+1},U_{2}^{n+1}) \|_{c}^2 -  \| (U_{1}^{n},U_{2}^{n}) \|_{c}^2 \right).
\end{split}
\end{equation}
Therefore, we can have, by combining the above inequalities with (\ref{eq:Thm_1_main_eqs}),
\begin{equation}
\begin{split}
&\dfrac{1}{\tau} \sum_{i=1}^{2} \| U_{i}^{n+1}-U_{i}^{n} \|_{m_{ii}}^2 + \| (U_{1}^{n+1},U_{2}^{n+1}) \|_{a}^2 + \| (U_{1}^{n+1},U_{2}^{n+1}) \|_{c}^2 \\
\leq
&\dfrac{1}{\tau}\gamma^2 \sum_{i=1}^{2} \| U_{i}^{n}-U_{i}^{n-1} \|_{m_{ii}}^2 + \| (U_{1}^{n},U_{2}^{n}) \|_{a}^2 + \| U_{2}^{n+1} - U_{2}^{n} \|_{a_{22}}^2 + \| (U_{1}^{n},U_{2}^{n}) \|_{c}^2.
\end{split}
\end{equation}
If $\tau \leq (1-\gamma^2) \inf_{W\in V_{2,H}}\cfrac{\|W\|_{m_{22}}^{2}}{\|W\|_{a_{22}}^{2}}$, then 
\begin{equation}
\begin{split}
\dfrac{1}{\tau} (1-\gamma^2) \sum_{i=1}^{2} \| U_{i}^{n+1}-U_{i}^{n} \|_{m_{ii}}^2
\geq \| U_{2}^{n+1} - U_{2}^{n} \|_{a_{22}}^2,
\end{split}    
\end{equation}
and thus
\begin{equation}
\begin{split}
    &\dfrac{\gamma^2}{\tau} \sum_{i=1}^{2} \| U_{i}^{n+1}-U_{i}^{n} \|_{m_{ii}}^2 + \| (U_{1}^{n+1},U_{2}^{n+1}) \|_{a}^2 + \| (U_{1}^{n+1},U_{2}^{n+1}) \|_{c}^2 \\
    \leq
    &\dfrac{\gamma^2}{\tau} \sum_{i=1}^{2} \| U_{i}^{n}-U_{i}^{n-1} \|_{m_{ii}}^2 + \| (U_{1}^{n},U_{2}^{n}) \|_{a}^2 + \| (U_{1}^{n},U_{2}^{n}) \|_{c}^2.
\end{split}
\end{equation}
\end{proof}

Similarly, we can also prove the stability result for the discretization scheme 2.
\begin{thm}
For the discretization scheme 2 in (\ref{eq:Scheme_2}), we have the following stability result
\begin{equation}
\begin{split}
    &\dfrac{\gamma^2}{\tau} \sum_{i=1}^{2} \| U_{i}^{n+1}-U_{i}^{n} \|_{m_{ii}}^2 + \| (U_{1}^{n+1},U_{2}^{n+1}) \|_{a}^2 + \| (U_{1}^{n+1},U_{2}^{n+1}) \|_{c}^2 \\
    \leq
    &\dfrac{\gamma^2}{\tau} \sum_{i=1}^{2} \| U_{i}^{n}-U_{i}^{n-1} \|_{m_{ii}}^2 + \| (U_{1}^{n},U_{2}^{n}) \|_{a}^2 + \| (U_{1}^{n},U_{2}^{n}) \|_{c}^2 ,
\end{split}
\end{equation}
if the stability condition
\begin{equation}
\label{eq:Scheme_2_stability}
    \tau \leq (1-\gamma^2) \inf_{W\in V_{2,H}}\cfrac{\|W\|_{m_{22}}^{2}}{\|W\|_{a_{22}}^{2}+\|W\|_{c_{22}}^{2}}
\end{equation}
is satisfied.
\end{thm}

Finally, we remark that for two-value fields, we typically choose $I_1$ and $I_2$ to be the index sets for the continua corresponding to the high-value and low-value regions respectively. In this case, the stability conditions of the aforementioned partially explicit schemes are expected to be independent of the contrast of $\kappa$, if its low value is fixed. 
Indeed, this follows from the construction of cell problems in multicontinuum homogenization and can be proved similarly to Lemma 4.1 in \cite{chung2021contrast}.
Recall that by (\ref{eq:mc_cell_problem_1}), the energy norm of $\phi_i$ is minimized under the constraint that it behaves as a constant in the continuum on average.
Therefore, $\phi_i$ is independent of the contrast of $\kappa$ and is constant in the high-value regions. 
However, since the constraint in (\ref{eq:mc_cell_problem_2}) requires $\phi_i^m$ to represent the linear function on average, it can be seen that $\phi_i^m$ for $i\in I_2$ remains independent of the contrast and is constant in the high-value regions, whereas $\phi_i^m$ for $i\in I_1$ does not exhibit this property.
As a consequence, the $m_{22}$-norm, the $c_{22}$-norm and the $a_{22}$-norm are all independent of the contrast, leading to contrast-independent stability conditions.
This result is confirmed by numerical examples. We also note that the stability conditions remain contrast-independent in more complex cases, provided the continua are chosen appropriately.

\section{Optimized decomposition of multicontinuum space}
\label{sec:construction}

In this section, we propose possible methods to reconstruct $V_1$ and $V_2$ such that the stability conditions are relaxed and their explicit forms can be given.
A Rayleigh quotient problem involving tensors is formulated and some simplifications are applied to facilitate the computation.

To begin with, given auxiliary functions $\psi_i$, the multiscale basis functions $\phi_i$ and $\tilde{\phi}_i$ can be obtained from cell problems (\ref{eq:mc_cell_problem_1}) and (\ref{eq:mc_cell_problem_2}). Compared with the previous decomposition (\ref{eq:first_decomposition}) in Section \ref{sec:schemes}, we set $I_1=\{i_0+1,i_0+2,\ldots,N\}$ and $I_2=\{1,2,\ldots,i_0\}$, mix the multiscale basis functions, and redefine
\begin{equation}
\label{eq:Reconstruction_v}
\begin{split}
V_1 = \left\{ \sum_{x_{l} \in I_H} \mathds{1}_{K_{H}(x_l)} \sum_{i > i_{0}} \left( \sum_j v_{i,j} \phi_{j}  \hat{U}_{i} + \sum_j v_{i,j} \tilde{\phi}_{j}\cdot\nabla \hat{U}_{i} \right): \; \hat{U}_i \in \mathcal{F}(\Omega) \right\}, \\
V_2 = \left\{ \sum_{x_{l} \in I_H} \mathds{1}_{K_{H}(x_l)} \sum_{i \leq i_{0}} \left( \sum_j v_{i,j} \phi_{j}  \hat{U}_{i} + \sum_j v_{i,j} \tilde{\phi}_{j}\cdot\nabla \hat{U}_{i} \right): \;\hat{U}_i \in \mathcal{F}(\Omega) \right\}.
\end{split}
\end{equation}
Here, $v_i=(v_{i,j})_{1\leq j \leq N} \in \mathbb{R}^N$ are linearly independent, and $i_{0} \in [1,N]$ is an integer to be determined. They will be carefully constructed such that the stability conditions (\ref{eq:Scheme_1_stability}) and (\ref{eq:Scheme_2_stability}) hold. The definitions of the maps $T_1$ and $T_2$, as well as the bilinear forms $m_{ij}(\cdot,\cdot)$, $a_{ij}(\cdot,\cdot)$, $b_{ij}(\cdot,\cdot)$ and $c_{ij}(\cdot,\cdot)$, are changed accordingly. Besides, the macroscopic variables $\hat{U}_i$ are linear combinations of the variables $U_i$ induced by $\psi_i$, with $\hat{U}_i = \sum_j \hat{v}_{i,j} U_j$. Denoting $v=(v_{i,j})_{N \times N}$ and $\hat{v}=(\hat{v}_{i,j})_{N \times N}$, we have $\hat{v}=(v^{T})^{-1}$ to ensure that $V_{mc}=V_1+V_2$. In the discussions that follow, we use $U_i$ in place of $\hat{U}_i$ for brevity of notation.

We consider $V_{i,H}=\prod_{j\in I_{i}}P(\mathcal{T}_{H})$, where $P(\mathcal{T}_{H})$
is the piecewise polynomial finite element space (with zero trace on $\partial \Omega$) with 
\begin{equation}
\sup_{W\in V_{2,H}}\dfrac{\|W\|_{H^{1}}^{2}}{\|W\|_{L^{2}}^{2}} \leq \dfrac{C_1}{H^{2}}    
\end{equation}
for some constant $C_1$, which depends on the shape regularity of the coarse mesh $\mathcal{T}_{H}$ but is independent of the mesh size $H$. Indeed, for a given mesh $\mathcal{T}_{H}$, $C_1$ can be obtained by solving a generalized eigenvalue problem about the stiffness and the mass matrix.

\subsection{Construction for discretization scheme 1}

For the construction for discretization scheme 1, we can formulate a min-max optimization problem to obtain $\{v_i\}_{1\leq i\leq N}$ and $i_{0}$ in the general case. However, since this problem is in tensor form, it is complicated to solve directly. To address this, by introducing reasonable assumptions, we can simplify the problem and provide suboptimal solutions.

To promise the stability condition (\ref{eq:Scheme_1_stability}), we can estimate the upper bound for $\sup_{W\in V_{2,H}}\frac{\|W\|_{a_{22}}^{2}}{\|W\|_{m_{22}}^{2}}$ as follows
\begin{equation}
\label{eq:Scheme_1_estimate_bound}
\sup_{W\in V_{2,H}}\frac{\|W\|_{a_{22}}^{2}}{\|W\|_{m_{22}}^{2}}
\leq \sup_{W\in V_{2,H}}\frac{\|W\|_{a_{22}}^{2}}{\|W\|_{H^1}^{2}} \cdot \sup_{W\in V_{2,H}}\frac{\|W\|_{H^1}^{2}}{\|W\|_{L^2}^{2}} \cdot \sup_{W\in V_{2,H}}\frac{\|W\|_{L^2}^{2}}{\|W\|_{m_{22}}^{2}}.
\end{equation}

For $1 \leq k,l \leq N$, we define a $d\times d$ matrix $A^{kl}$ by $A^{kl}=\dfrac{1}{|K|}a_{K}(\tilde{\phi}_k, \tilde{\phi}_l)$, or equivalently by
\begin{equation}
A_{mn}^{kl}=\dfrac{1}{|K|}a_{K}(\phi_k^m, \phi_l^n)
= \dfrac{1}{|K|} \int_{K} \kappa \nabla\phi_k^m \cdot \nabla\phi_l^n,
\end{equation}
where $K$ is a coarse block in the given mesh $\mathcal{T}_H$.
Then we can define a rank-4 tensor $A$ by $A_{kmln}=A_{mn}^{kl}$ for $1 \leq k,l \leq N$ and $1 \leq m,n \leq d$. It can be noticed that $A_{kmln}=A_{lnkm}$, and $A^{kl}$ is nearly diagonal. Also, we define an $N \times N$ symmetric matrix $M$ by 
\begin{equation}
M_{kl} = \frac{1}{|K|} \int_K \phi_k \phi_l.
\end{equation}
From the constraint in the cell problem (\ref{eq:mc_cell_problem_1}), it can be observed that $M$ is a diagonally dominant matrix with diagonal entries near $1$ and off-diagonal entries being small.

Now we consider the following Rayleigh quotient problem involving tensors: for $1 \leq i \leq N$, we want to find a subspace $S_{i} \subset \mathbb{R}^N$ and a number $\lambda_i \in \mathbb{R}$ such that 
\begin{equation}
\label{eq:Scheme_1_Rayleigh}
\begin{split}
S_i & = \argmin_{\substack{S\subset\mathbb{R}^N \\ \dim(S)=i}} \max_{v\otimes w\in S\otimes\mathbb{R}^d} 
\dfrac{(A:v\otimes w):(v\otimes w)}{((M\otimes I):v\otimes w):(v\otimes w)}, \\
\lambda_i & = \min_{\substack{S\subset\mathbb{R}^N \\ \dim(S)=i}} \max_{v\otimes w\in S\otimes\mathbb{R}^d}
\dfrac{(A:v\otimes w):(v\otimes w)}{((M\otimes I):v\otimes w):(v\otimes w)}.
\end{split}
\end{equation}
Here, $\otimes$ and $:$ denote the tensor product and the double dot product respectively. 
For simplicity, we assume that one can obtain similar splittings for different coarse blocks, so the above problem only needs to be solved once. Otherwise, the formulated quotient problem would need to be global across the entire spatial region rather than localized within a coarse block.
We let $i_{0}$ represent the number of small $\lambda_i$'s, and select an orthonormal basis $\{v_i\}_{1 \leq i \leq N}$ of $\mathbb{R}^N$ with respect to $M$, namely, $v_i^T M v_j = \delta_{ij}$, such that
\begin{equation}
\label{eq:Scheme_1_v_i_construction}
S_{i_{0}} = \operatorname*{span}_{1 \leq i \leq i_{0}} \{v_i\},\quad
\mathbb{R}^N = S_{i_{0}} \oplus \operatorname*{span}_{i_{0}+1 \leq i \leq N} \{v_i\}.
\end{equation}
With such construction, the $m_{22}$-norm is nearly identical to the $L^2$-norm restricted in $V_{2,H}$.

\begin{lem}
For any $W\in V_{2,H}$, we have 
\begin{equation}
\|W\|_{L^2} \approx \|W\|_{m_{22}},
\end{equation}
if $\{v_i\}_{1 \leq i \leq N}$ is orthonormal with respect to $M$.
\end{lem}

\begin{proof}
For any $W \in  V_{2,H}$, we have $W=(U_1, U_2, \ldots, U_{i_{0}})$ for some $U_i \in P(\mathcal{T}_{H})$, and
\begin{equation}
\begin{split}
\|W\|_{m_{22}}^{2}
& \approx (\sum_{x_{l} \in I_H} \mathds{1}_{K_{H}(x_l)} \sum_{i \leq i_{0}} \sum_k v_{i,k} \phi_{k}  U_{i}, \; \sum_{x_{l} \in I_H} \mathds{1}_{K_{H}(x_l)} \sum_{j \leq i_{0}} \sum_l v_{j,l} \phi_{l}  U_{j}) \\
& \approx \int_{\Omega} \sum_{i,j \leq i_{0}} \left( \sum_{k,l} v_{i,k} M_{kl} v_{j,l} \right) U_{i} U_{j} \\
& = \int_{\Omega}
\begin{pmatrix}
    U_1, U_2, \cdots, U_{i_{0}}
\end{pmatrix}
\begin{pmatrix}
    v_1^T \\ v_2^T \\ \vdots \\ v_{i_{0}}^T
\end{pmatrix}
M
\begin{pmatrix}
    v_1, v_2, \cdots, v_{i_{0}}
\end{pmatrix}
\begin{pmatrix}
    U_1 \\ U_2 \\ \vdots \\ U_{i_{0}} 
\end{pmatrix}
\\
& = \|W\|_{L^2}^2
,
\end{split}
\end{equation}
where the first and the second approximations are attributed to the estimates (\ref{eq:mc_scalings}) and the smoothness of $U_i$ respectively. The $M$-orthonormality of $\{v_i\}_{1 \leq i \leq N}$ is applied in the final equality.
\end{proof}

Moreover, one can similarly show that $m_{21}(W_1,W_2)=(T_1 W_1,T_2 W_2)=0$ for any $W_1 \in V_{1,H}$ and $W_2 \in V_{2,H}$, and thus $\gamma=0$. Then we can obtain the following stability condition in an explicit form.

\begin{thm}
\label{lem:Scheme_1_construction}
The discretization scheme 1 described in (\ref{eq:Scheme_1}) is stable if we choose 
\begin{equation}
\tau \leq C_1^{-1} \lambda_{i_0}^{-1} H^2
\end{equation}
\end{thm}
and construct $\{v_i\}_{1 \leq i \leq N}$ as in (\ref{eq:Scheme_1_Rayleigh}) and (\ref{eq:Scheme_1_v_i_construction}).

\begin{proof}
We start our proof by estimating $\sup_{W\in V_{2,H}}\dfrac{\|W\|_{a_{22}}^{2}}{\|W\|_{H^1}^{2}}$. For any $W \in  V_{2,H}$, we have $W=(U_1,\ldots, U_{i_{0}})$ for some $U_i \in P(\mathcal{T}_{H})$, and
\begin{equation}
\begin{split}
\|W\|_{a_{22}}^{2}
& = a(\sum_{x_{l} \in I_H} \mathds{1}_{K_{H}(x_l)} \sum_{i \leq i_{0}} \sum_k v_{i,k} \tilde{\phi}_{k} \cdot \nabla U_{i}, \; \sum_{x_{l} \in I_H} \mathds{1}_{K_{H}(x_l)} \sum_{j \leq i_{0}} \sum_l v_{j,l} \tilde{\phi}_{l} \cdot \nabla U_{j}) \\
& \approx \int_{\Omega} \sum_{i,j \leq i_{0}} (\nabla U_i)^T \left( \sum_{k,l} v_{i,k} A^{kl} v_{j,l} \right) \nabla U_j \\
& = \int_{\Omega} \sum_{i,j \leq i_{0}} (A : v_j \otimes \nabla U_j):(v_i \otimes \nabla U_i)
\\
& \leq \int_{\Omega} \lambda_{i_{0}} \sum_{1\leq i\leq i_{0}}(v_i^T M v_i)|\nabla U_i|^2 
= \int_{\Omega} \lambda_{i_{0}} \sum_{1\leq i\leq i_{0}}|\nabla U_i|^2,
\end{split}
\end{equation}
since $v_i^T M v_j = \delta_{ij}$. Therefore,  we obtain
\begin{equation}
\begin{split}
\sup_{W\in V_{2,H}}\frac{\|W\|_{a_{22}}^{2}}{\|W\|_{H^1}^{2}} 
\leq \lambda_{i_{0}}
\leq \dfrac{H^2}{C_1 \tau},
\end{split}
\end{equation}
and we have 
\begin{equation}
\sup_{W\in V_{2,H}}\cfrac{\|W\|_{a_{22}}^{2}}{\|W\|_{m_{22}}^{2}}
\leq \sup_{W\in V_{2,H}}\cfrac{\|W\|_{a_{22}}^{2}}{\|W\|_{H^1}^{2}} \cdot \sup_{W\in V_{2,H}}\cfrac{\|W\|_{H^1}^{2}}{\|W\|_{L^2}^{2}} \cdot \sup_{W\in V_{2,H}}\cfrac{\|W\|_{L^2}^{2}}{\|W\|_{m_{22}}^{2}}
\leq \dfrac{1}{\tau}
= \dfrac{(1-\gamma^2)}{\tau},
\end{equation}
which promises the stability condition (\ref{eq:Scheme_1_stability}).
\end{proof}

To solve the tensor-based min-max problem (\ref{eq:Scheme_1_Rayleigh}) more efficiently, we can introduce some assumptions to simplify it. One of the possible methods is to restrict the choices for $S_i$ to be finite based on some a priori information. For instance, we can assume that $S_i$ is spanned by $i$ natural basis vectors in $\mathbb{R}^N$, giving a total of $\binom{N}{i}$ choices for $S_i$. Under this assumption, the optimization problem becomes much easier to solve. When the $M$-inner product is sufficiently close to the $l^2$-inner product, i.e., when $M$ is almost diagonal, this assumption implies that the continua are divided into two groups as the multicontinuum space is decomposed.

Another method is to follow the idea of eigendecomposition and this is what we will apply in the later numerical examples. For the rank-4 tensor $A$, we define a rank-reduced tensor $\tilde{A}$ such that $\tilde{A}_{kl}=\operatorname*{max\;eig}(A^{kl})$, and then $\tilde{A}$ can be considered as a symmetric $N \times N$ matrix. We perform the generalized eigenvalue decomposition of $(\tilde{A},M)$ and select $\{v_i\}_{1 \leq i \leq N}$ to be the $M$-orthonormal eigenvectors of $\tilde{A}$ with the corresponding eigenvalues $\lambda_i$ ascending such that
\begin{equation}
\label{eq:Scheme_1_eig}
    \tilde{A} v_i=\lambda_i M v_i.
\end{equation}
The number of small eigenvalues is again denoted by $i_0$. Then choosing $\tau \leq C_1^{-1} \lambda_{i_0}^{-1} H^2$ can ensure the stability under the assumption that
\begin{equation}
\sum_{k,l} \operatorname{max\;eig}\left( v_{i,k} A^{kl} v_{j,l} \right)
\leq
\sum_{k,l} \left( v_{i,k} \operatorname{max\;eig}\left( A^{kl} \right)v_{j,l} \right),
\end{equation}
for any $i,j\leq i_{0}$.
We can follow the idea of the proof in Lemma \ref{lem:Scheme_1_construction} to show it. We remark that the above assumption may not hold when $v_i$ has negative components. However, it holds when the diagonal parts of $A$ are more dominant than the off-diagonal parts, which often occurs in practical examples. We can directly verify whether the assumption holds after solving the generalized eigenvalue problem. If it does not, then we can return to using the first method.

\subsection{Construction for discretization scheme 2}

For discretization scheme 2, we may also formulate a Rayleigh quotient problem similar to the one in (\ref{eq:Scheme_1_Rayleigh}) for scheme 1 and apply some simplifications. However, since the effect of the $a_{22}$-norm is indeed more dominant than the $c_{22}$-norm, one can directly construct $\{v_i\}_{1\leq i\leq N}$ and $i_{0}$ the same as in (\ref{eq:Scheme_1_v_i_construction}). An explicit form of the stability condition is provided in the following.

We define an $N\times N$ symmetric matrix $C$ by
\begin{equation}
C_{ij} = \dfrac{1}{|K|} a_{K}(\phi_i,\phi_j) = \dfrac{1}{|K|} \int_K \kappa \nabla \phi_i \cdot \nabla \phi_j.
\end{equation}
Then for any $W \in  V_{2,H}$, we can write $W=(U_1,\ldots, U_{i_{0}})$ for some $U_i \in P(\mathcal{T}_{H})$. Therefore, we have
\begin{equation}
\begin{split}
\|W\|_{c_{22}}^{2}
& = a(\sum_{x_{l} \in I_H} \mathds{1}_{K_{H}(x_l)} \sum_{i \leq i_{0}} \sum_k v_{i,k} \phi_{k}  U_{i}, \; \sum_{x_{l} \in I_H} \mathds{1}_{K_{H}(x_l)} \sum_{j \leq i_{0}} \sum_l v_{j,l} \phi_{l}  U_{j}) \\
& = \int_{\Omega} \sum_{i,j \leq i_{0}} \left( \sum_{k,l} v_{i,k} C_{kl} v_{j,l} \right) U_{i} U_{j} \\
& = \int_{\Omega}
\begin{pmatrix}
    U_1, U_2, \cdots, U_{i_{0}}
\end{pmatrix}
\begin{pmatrix}
    v_1^T \\ v_2^T \\ \vdots \\ v_{i_{0}}^T
\end{pmatrix}
C
\begin{pmatrix}
    v_1, v_2, \cdots, v_{i_{0}}
\end{pmatrix}
\begin{pmatrix}
    U_1 \\ U_2 \\ \vdots \\ U_{i_{0}}
\end{pmatrix}
,
\end{split}
\end{equation}
and
\begin{equation}
\|W\|_{L^2}^2 = \int_{\Omega} 
\begin{pmatrix}
    U_1, U_2, \cdots, U_{i_{0}}
\end{pmatrix}
\begin{pmatrix}
    v_1^T \\ v_2^T \\ \vdots \\ v_{i_{0}}^T
\end{pmatrix}
M
\begin{pmatrix}
    v_1, v_2, \cdots, v_{i_{0}}
\end{pmatrix}
\begin{pmatrix}
    U_1 \\ U_2 \\ \vdots \\ U_{i_{0}}
\end{pmatrix}
.
\end{equation}
Note that
$\begin{pmatrix}
    U_1, U_2, \cdots, U_{i_{0}}
\end{pmatrix}
\begin{pmatrix}
    v_1^T \\ v_2^T \\ \vdots \\ v_{i_{0}}^T
\end{pmatrix}$
is a $1\times N$ matrix, and we obtain
\begin{equation}
\begin{split}
\sup_{W\in V_{2,H}}\cfrac{\|W\|_{c_{22}}^{2}}{\|W\|_{m_{22}}^{2}}
\leq \sup_{W\in V_{2,H}}\cfrac{\|W\|_{c_{22}}^{2}}{\|W\|_{L^2}^{2}} \cdot \sup_{W\in V_{2,H}}\cfrac{\|W\|_{L^2}^{2}}{\|W\|_{m_{22}}^{2}}
\leq \dfrac{\operatorname{max\;eig}(H^2 C,M)}{H^2}.
\end{split}
\end{equation}
Here, $\operatorname{max\;eig}(H^2 C,M)$ represents the largest eigenvalue for the generalized eigenvalue problem $H^2 C v= \lambda M v$.
Combining the stability condition (\ref{eq:Scheme_2_stability}) with
\begin{equation}
\begin{split}
\inf_{W\in V_{2,H}}\cfrac{\|W\|_{m_{22}}^{2}}{\|W\|_{a_{22}}^{2}+\|W\|_{c_{22}}^{2}}
& \geq \left( \sup_{W\in V_{2,H}}\cfrac{\|W\|_{a_{22}}^{2}}{\|W\|_{m_{22}}^{2}} + \sup_{W\in V_{2,H}}\cfrac{\|W\|_{c_{22}}^{2}}{\|W\|_{m_{22}}^{2}} \right)^{-1} \\
& \geq (C_1 \lambda_{i_0} + \operatorname{max\;eig}(H^2 C,M))^{-1} H^2,    
\end{split}
\end{equation}
we conclude that
\begin{equation}
\tau \leq (C_1 \lambda_{i_0} + \operatorname{max\;eig}(H^2 C,M))^{-1} H^2
\end{equation}
will guarantee the stability. As we discussed before, $\operatorname{max\;eig}(H^2 C,M)$ is expected to be independent of the contrast of $\kappa$, since $\phi_i$ is contrast-independent and is constant in the high-value regions, if the continua are chosen appropriately.

Finally, we remark that the new continua obtained through the space decomposition are independent of the choice of basis, as long as $\operatorname*{span}_{i}\{\psi_i\}$ remains the same, and thus represent physical quantities. This is due to the construction of the cell problems (\ref{eq:mc_cell_problem_1}) and (\ref{eq:mc_cell_problem_2}), as well as the Rayleigh quotient problem (\ref{eq:Scheme_1_Rayleigh}). The former ensures that $\operatorname*{span}_{i}\{\phi_i\}$ and $\operatorname*{span}_{i}\{\phi_i^m\}$ are only dependent on $\operatorname*{span}_{i}\{\psi_i\}$, while the latter ensures that the solution of the quotient problem is only dependent on $\operatorname*{span}_{i}\{\phi_i\}$ and $\operatorname*{span}_{i}\{\phi_i^m\}$.
Besides, in the discussions of this section, we derive the stability conditions on the time step size $\tau$ by first determining the critical eigenvalue $\lambda_{i_0}$. The procedure can also be reversed: given a desired step size $\tau$, we can pick the largest eigenvalue $\lambda_{i_0(\tau,H^2)}$ for the space decomposition such that the scheme is stable, assuming $C_1$ is provided.

\section{Numerical examples}
\label{sec:numerical}

In this section, we present some numerical examples to show the accuracy and stability of our proposed method. 
We consider the parabolic PDE
\begin{equation}
    \partial_t u(x,t) - \nabla \cdot (\kappa(x) \nabla u(x,t)) = f(x), \quad x \in \Omega, \quad t\in (0,T],
\end{equation}
with zero Dirichlet boundary condition and zero initial condition, where $\Omega = (0,1)^2$, $\kappa$ is of high contrast, and $f(x)=10 e^{-40 \left( (x_1-0.5)^2+(x_2-0.5)^2 \right)}$ for $x=(x_1,x_2)\in \Omega$ as shown in Figure \ref{fig:source_term}.
We denote the coarse mesh size by $H$ and divide the computational domain $\Omega$ into $1/H \times 1/H$ square coarse blocks $K$'s of the same size. The oversampled domain $K^{+}$ is defined as an extension of $K$ itself by $l$ layers of coarse blocks. According to the analysis in \cite{chung2018constraint, efendiev2023multicontinuum}, we will take $l=\lceil -2\ln(H) \rceil$ in the following examples to mitigate boundary effects. We define the relative $L^2$ errors of solution in continuum $i$ at a specific time $t$ by
\begin{equation}
\label{eq:def_relative_error}
    e^{(i)}(t)=\sqrt{ \frac{\sum_K|\frac1{|K|}\int_K U_i(x,t) dx -  \frac{1}{\int_{K}\psi_i(x) dx}\int_{K}u(x,t)\psi_i(x) dx|^2}{\sum_K|\frac{1}{\int_{K}\psi_i(x) dx}\int_{K}u(x,t)\psi_i(x) dx|^2} },
\end{equation}
where $K$ denotes the coarse block, $U_i$ is the multiscale solution in continuum $i$, and $u$ is the reference solution solved on the fine grid. We fix the fine mesh size $h$ to be $1/400$.

\begin{figure}
    \centering
    \includegraphics[width=7cm]{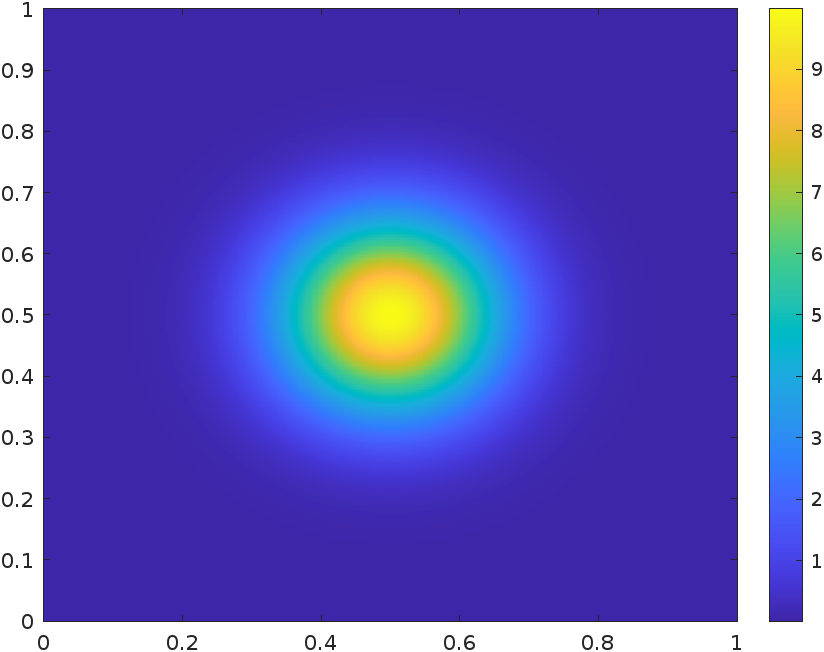}
    \caption{Source term $f$}
    \label{fig:source_term}
\end{figure}

\subsection{Two-continuum cases}

In the following two examples, we consider two-continuum cases, where the auxiliary functions $\psi_1$ and $\psi_2$ are chosen as the characteristic functions in the regions of low and high values respectively. We decompose the multicontinuum space as in (\ref{eq:first_decomposition}) and choose $I_1=\{2\}$ and $I_2=\{1\}$, as remarked at the end of Section \ref{sec:schemes}.
Indeed, if we apply the optimized decomposition from Section \ref{sec:construction} and solve the eigenvalue problem (\ref{eq:Scheme_1_eig}), the eigenvector corresponding to the small eigenvalue is $e_1=(1,0)^T$. This suggests that the first continuum is identified as slow dynamics, which aligns with our intuition. The numerical results of the eigenvalue problems are omitted in these two examples for brevity. More diverse results will be discussed in the multiple-continuum cases later.

\subsubsection{Example 1}

We consider a two-continuum case as shown in Figure \ref{fig:Example_1_setting}, where 
\begin{equation}
\kappa(x)=
\begin{cases}
1,&\quad x \in \Omega_1, \\
\max \kappa, & \quad x \in \Omega_2,
\end{cases}
\end{equation}
and $\Omega_1$ and $\Omega_2$ are respectively blue and yellow regions.

\begin{figure}
    \centering
    \includegraphics[width=7cm]{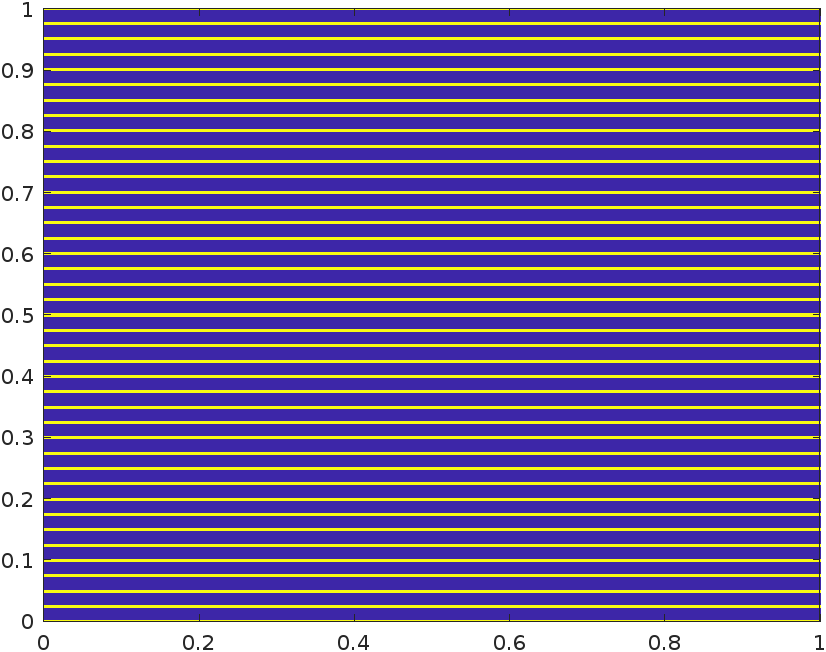}
    \quad
    \includegraphics[width=7cm]{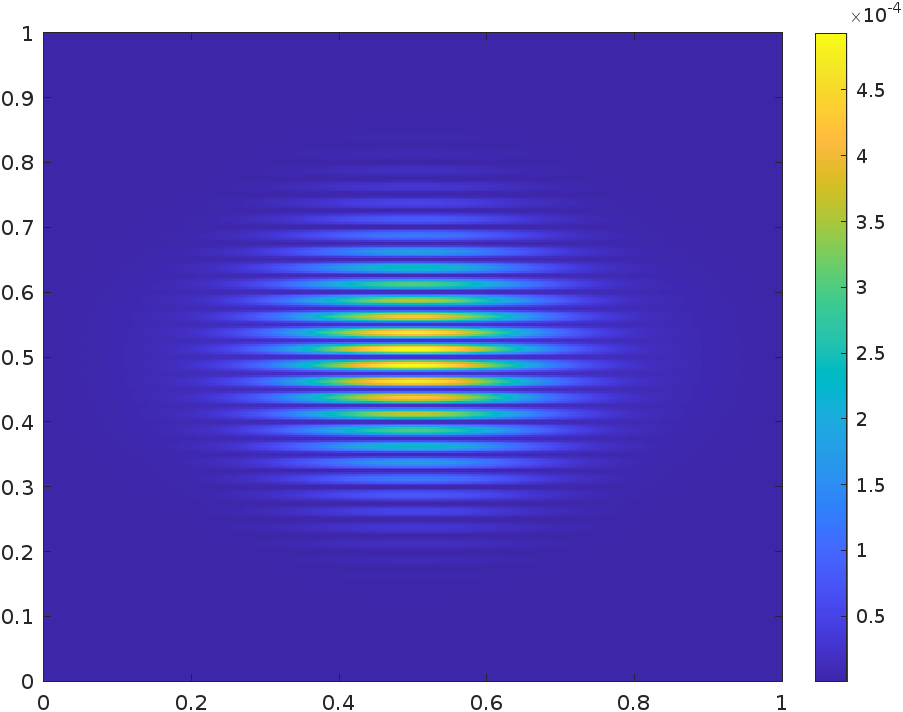}
    \caption{Left: Layered field $\kappa$ in Example 1. Right: Reference solution at the final time $T$ when $\max \kappa = 10^5$ in Example 1.}
    \label{fig:Example_1_setting}
\end{figure}

The estimates of the stability conditions for discretization scheme 1 are computed and listed in Table \ref{tab:Example_1_condition_1}. It can be observed that the infimum of $\frac{\|W\|_{m_{22}}^2}{\|W\|_{a_{22}}^2}$ over $V_{2,H}$ is much larger than that over $V_{1,H}$ and is independent of the contrast of $\kappa$. This indicates that the partially explicit scheme significantly relaxes the stability condition and reduces computational cost, especially for high-contrast cases. Moreover, we notice that the infimum decreases as the coarse mesh size $H$ decreases. Similar observations hold for the stability conditions of scheme 2, as shown in Table \ref{tab:Example_1_condition_2}.

\begin{table}
\centering
\caption{Estimates of $\inf \frac{\|W\|_{m_{22}}^2}{\|W\|_{a_{22}}^2}$ for different coarse mesh sizes $H$ and contrasts $\frac{\max{\kappa}}{\min{\kappa}}$ in Example 1}
\begin{tabular}{|c|c|c|c|c|c|c|c|c|}
\hline
\multirow{2}*{H} & \multicolumn{2}{|c|}{$\frac{\max{\kappa}}{\min{\kappa}}=10^4$} & \multicolumn{2}{|c|}{$\frac{\max{\kappa}}{\min{\kappa}}=10^5$} & \multicolumn{2}{|c|}{$\frac{\max{\kappa}}{\min{\kappa}}=10^6$} & \multicolumn{2}{|c|}{$\frac{\max{\kappa}}{\min{\kappa}}=10^7$}\\
\cline{2-9}
    & $V_{1,H}$ & $V_{2,H}$ & $V_{1,H}$ & $V_{2,H}$ & $V_{1,H}$ & $V_{2,H}$ & $V_{1,H}$ & $V_{2,H}$ \\
\hline

$1/10$ & $4.54 \times 10^{-7}$ & $3.13 \times 10^{-4}$ & 
$4.55 \times 10^{-8}$ & $3.13 \times 10^{-4}$ & 
$4.55 \times 10^{-9}$ & $3.13 \times 10^{-4}$ & 
$4.55 \times 10^{-10}$ & $3.13 \times 10^{-4}$ \\

$1/20$ & $1.14 \times 10^{-7}$ & $1.57 \times 10^{-4}$ & 
$1.14 \times 10^{-8} $ & $1.57 \times 10^{-4}$ & 
$1.14 \times 10^{-9} $ & $1.57 \times 10^{-4}$ & 
$1.14 \times 10^{-10}$ & $1.57 \times 10^{-4}$ \\

$1/40$ & $2.85 \times 10^{-8}$ & $7.55 \times 10^{-5}$ & 
$2.85 \times 10^{-9} $ & $7.55 \times 10^{-5}$ & 
$2.85 \times 10^{-10}$ & $7.55 \times 10^{-5}$ & 
$2.85 \times 10^{-11}$ & $7.55 \times 10^{-5}$ \\

\hline
\end{tabular}
\label{tab:Example_1_condition_1}
\end{table}

\begin{table}
\centering
\caption{Estimates of $\inf \frac{\|W\|_{m_{22}}^2}{\|W\|_{a_{22}}^2+\|W\|_{c_{22}}^2}$ for different coarse mesh sizes $H$ and contrasts $\frac{\max{\kappa}}{\min{\kappa}}$ in Example 1}
\begin{tabular}{|c|c|c|c|c|c|c|c|c|}
\hline
\multirow{2}*{H} & \multicolumn{2}{|c|}{$\frac{\max{\kappa}}{\min{\kappa}}=10^4$} & \multicolumn{2}{|c|}{$\frac{\max{\kappa}}{\min{\kappa}}=10^5$} & \multicolumn{2}{|c|}{$\frac{\max{\kappa}}{\min{\kappa}}=10^6$} & \multicolumn{2}{|c|}{$\frac{\max{\kappa}}{\min{\kappa}}=10^7$}\\
\cline{2-9}
    & $V_{1,H}$ & $V_{2,H}$ & $V_{1,H}$ & $V_{2,H}$ & $V_{1,H}$ & $V_{2,H}$ & $V_{1,H}$ & $V_{2,H}$ \\
\hline

$1/10$ & $4.41 \times 10^{-7}$ & $3.51 \times 10^{-5}$ & 
$4.54 \times 10^{-8}$ & $3.51 \times 10^{-5}$ & 
$4.55 \times 10^{-9}$ & $3.51 \times 10^{-5}$ & 
$4.55 \times 10^{-10}$ & $3.51 \times 10^{-5}$ \\

$1/20$ & $1.13 \times 10^{-7}$ & $3.16 \times 10^{-5}$ & 
$1.14 \times 10^{-8} $ & $3.16 \times 10^{-5}$ & 
$1.14 \times 10^{-9} $ & $3.16 \times 10^{-5}$ & 
$1.14 \times 10^{-10}$ & $3.16 \times 10^{-5}$ \\

$1/40$ & $2.85 \times 10^{-8}$ & $2.60 \times 10^{-5}$ & 
$2.85 \times 10^{-9} $ & $2.60 \times 10^{-5}$ & 
$2.85 \times 10^{-10}$ & $2.60 \times 10^{-5}$ & 
$2.85 \times 10^{-11}$ & $2.60 \times 10^{-5}$ \\

\hline
\end{tabular}
\label{tab:Example_1_condition_2}
\end{table}

Now we consider $\max \kappa = 10^5$ as an example. We choose the final time to be $T=1.5 \times 10^{-4}$, at which the solution reaches an almost-steady state depicted in Figure \ref{fig:Example_1_setting}. The time step size is taken to be $\Delta t=10^{-7}$. (We remark that the step size can actually be much larger due to the smoothness of the macroscopic variables.) We present the relative $L^2$ errors for different schemes when $H=1/10$, $1/20$ and $1/40$ in Figures \ref{fig:Example_1_error_1}, \ref{fig:Example_1_error_2} and \ref{fig:Example_1_error_3}. The reference solution is computed on the fine grid.
We notice that the curves representing the fully implicit scheme based on multicontinuum homogenization (blue), discretization scheme 1 (red) and discretization scheme 2 (yellow) almost coincide, while the fully explicit one (not shown in the figures) immediately diverges. This indicates that the proposed partially explicit schemes can achieve results as accurate as the fully implicit one but with lower computational cost.
Additionally, as the coarse mesh size becomes finer, higher accuracy is realized. From Tables \ref{tab:Example_1_condition_1} and \ref{tab:Example_1_condition_2}, we note that the stability conditions on the time step size $\tau$ are dependent on the coarse mesh size $H$, with smaller $H$ requiring a smaller $\tau$. Therefore, a balance between accuracy and computational cost should be struck in practice.

\begin{figure}
  \centering
  \includegraphics[width=7cm]{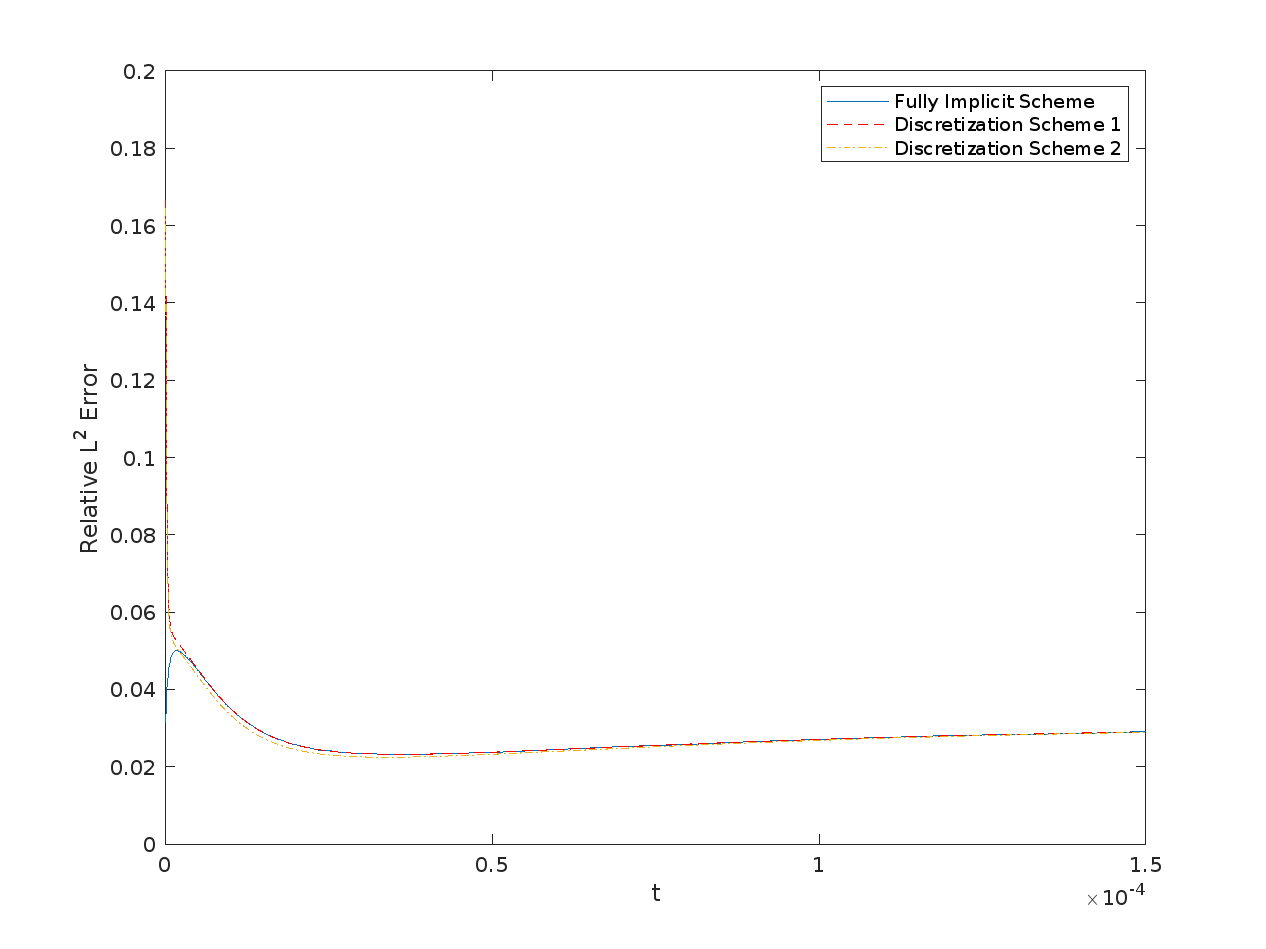}
  \quad
  \includegraphics[width=7cm]{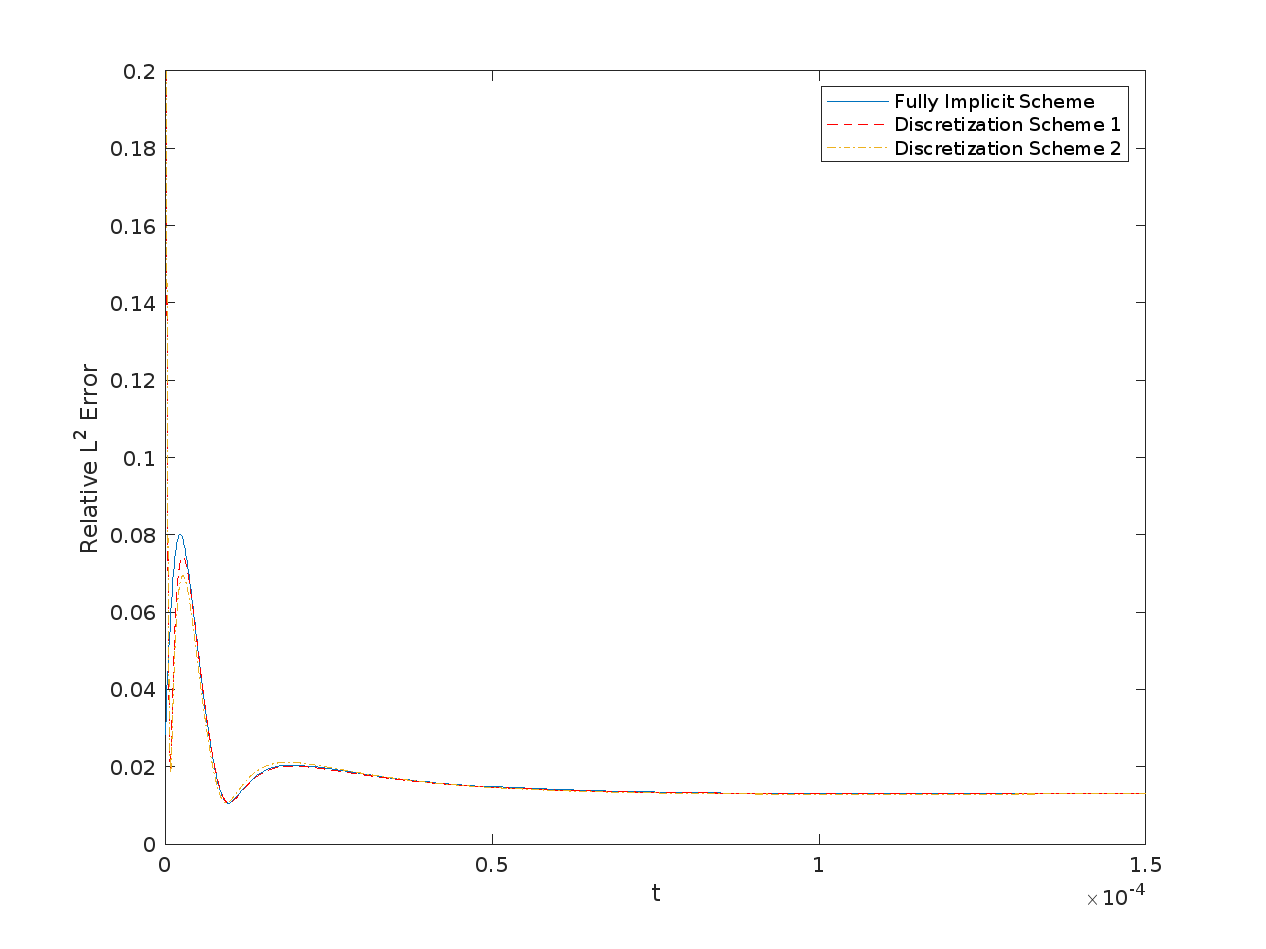}
  \quad
  \caption{Relative $L^2$ error for different schemes when $H=1/10$ and $l=5$ in Example 1. Left: $e^{(1)}(t)$. Right: $e^{(2)}(t)$.}
  \label{fig:Example_1_error_1}
\end{figure}

\begin{figure}
  \centering
  \includegraphics[width=7cm]{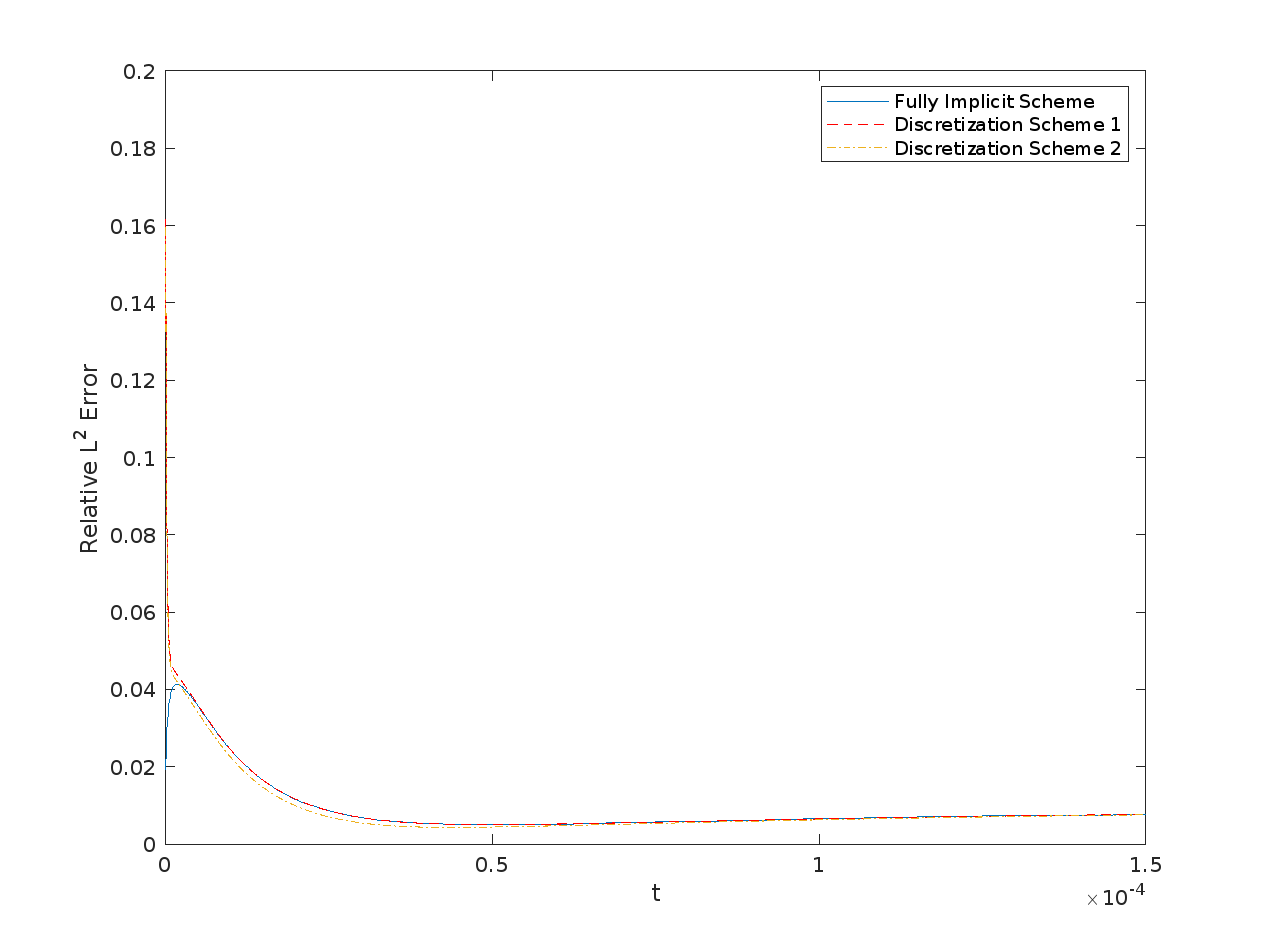}
  \quad
  \includegraphics[width=7cm]{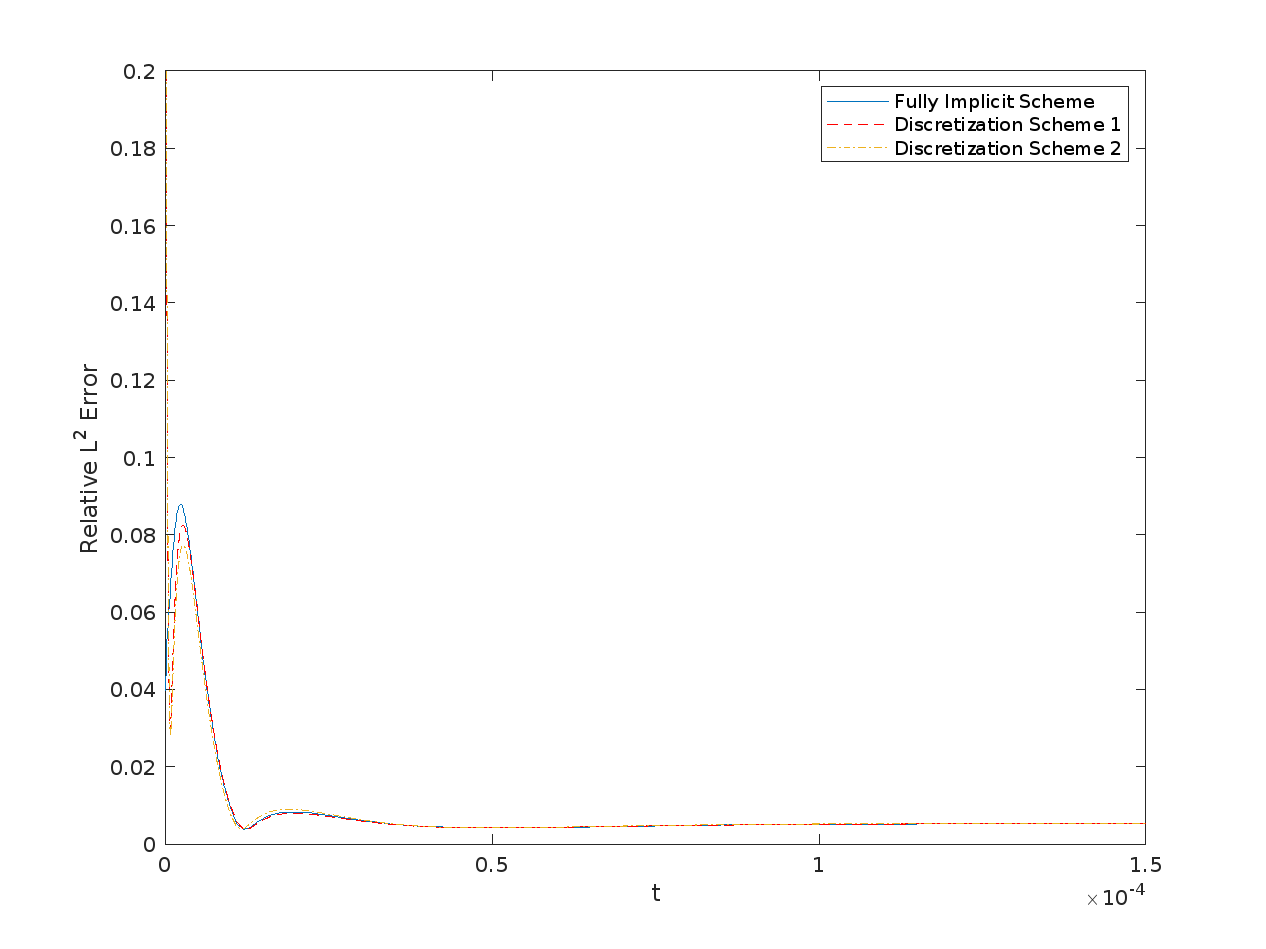}
  \quad
  \caption{Relative $L^2$ error for different schemes when $H=1/20$ and $l=6$ in Example 1. Left: $e^{(1)}(t)$. Right: $e^{(2)}(t)$.}
  \label{fig:Example_1_error_2}
\end{figure}

\begin{figure}
  \centering
  \includegraphics[width=7cm]{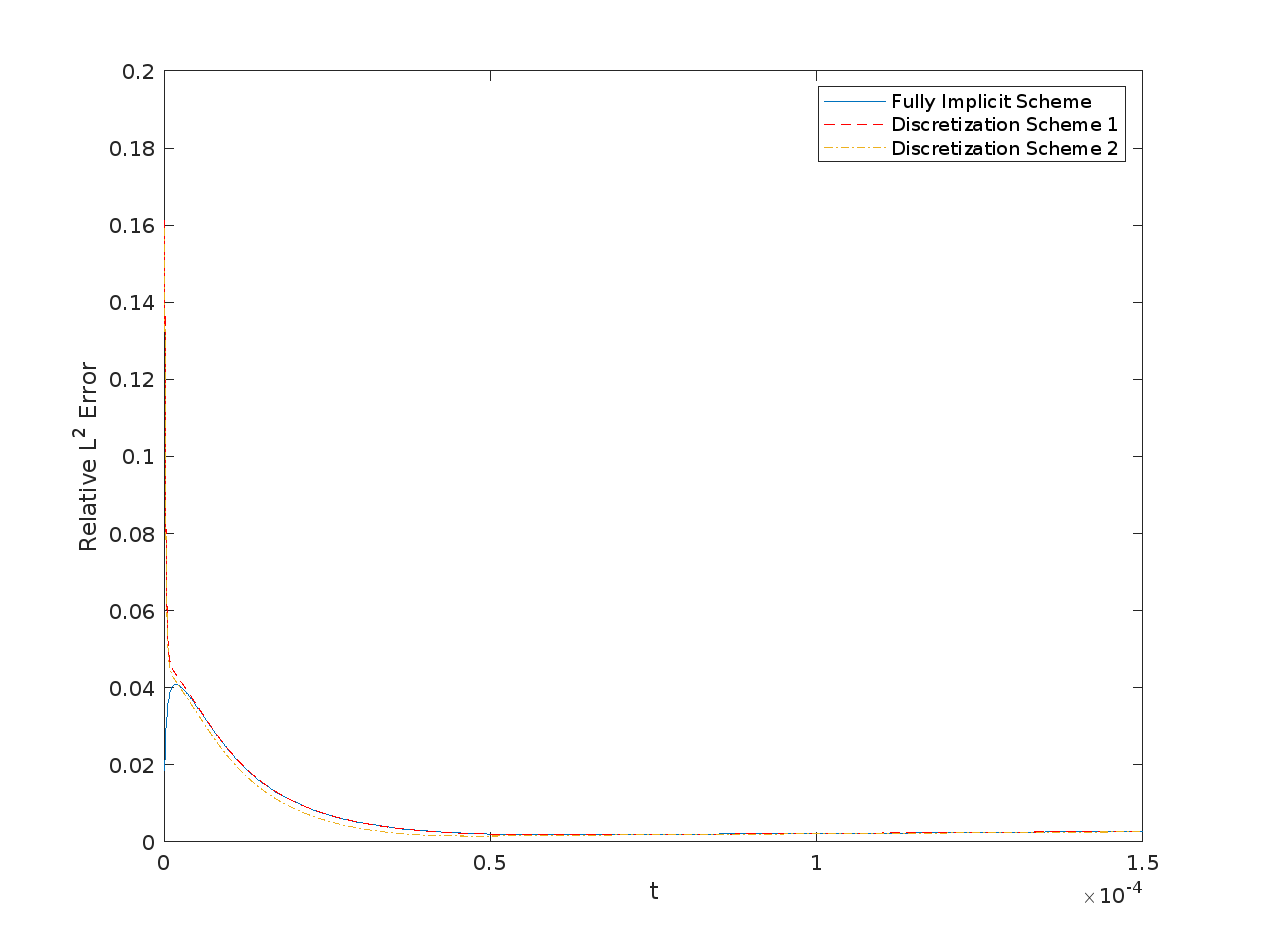}
  \quad
  \includegraphics[width=7cm]{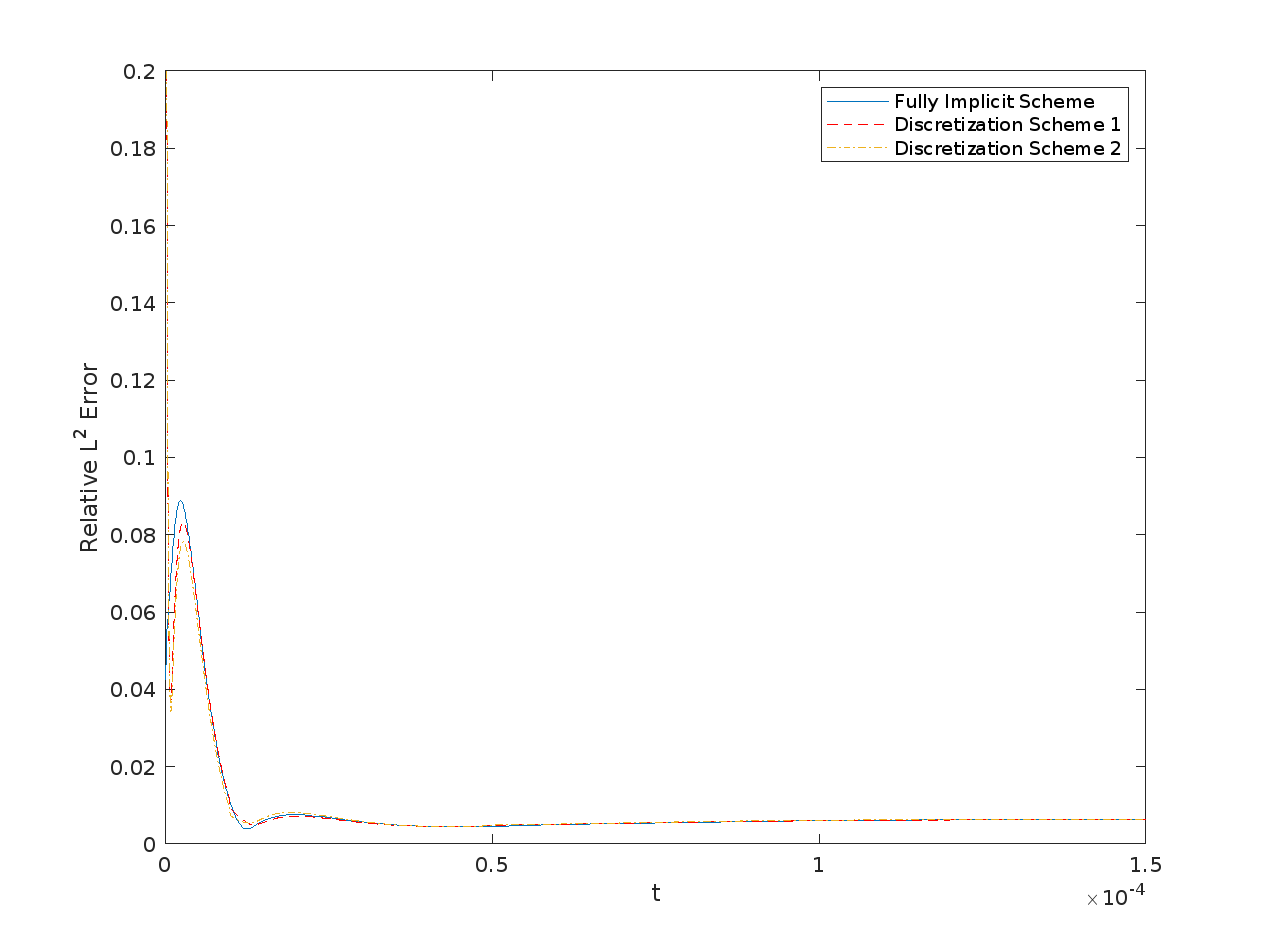}
  \quad
  \caption{Relative $L^2$ error for different schemes when $H=1/40$ and $l=8$ in Example 1. Left: $e^{(1)}(t)$. Right: $e^{(2)}(t)$.}
  \label{fig:Example_1_error_3}
\end{figure}

\subsubsection{Example 2}

In Example 2, we consider a non-periodic field $\kappa$ illustrated in Figure \ref{fig:Example_2_setting}. 
We estimate the stability conditions for scheme 1 and scheme 2, and list the results in Tables \ref{tab:Example_2_condition_1} and \ref{tab:Example_2_condition_2}. Similar observations can be made as in Example 1. The stability conditions are independent of the contrast of $\kappa$ as we analyzed. Next, we choose $\max \kappa = 10^5$, $T=10^{-4}$ and $\Delta t = 10^{-7}$. The reference solution at $T$ is also depicted in Figure \ref{fig:Example_2_setting}. The relative $L^2$ errors for different schemes with respect to the reference solution when $H=1/10$ and $1/20$ are presented in Figures \ref{fig:Example_2_error_1} and \ref{fig:Example_2_error_2}. Again, we observe that the errors for different schemes, except the fully explicit scheme, remain small and nearly identical, with the partially explicit schemes closely approximating the reference solution.

\begin{figure}
    \centering
    \includegraphics[width=7cm]{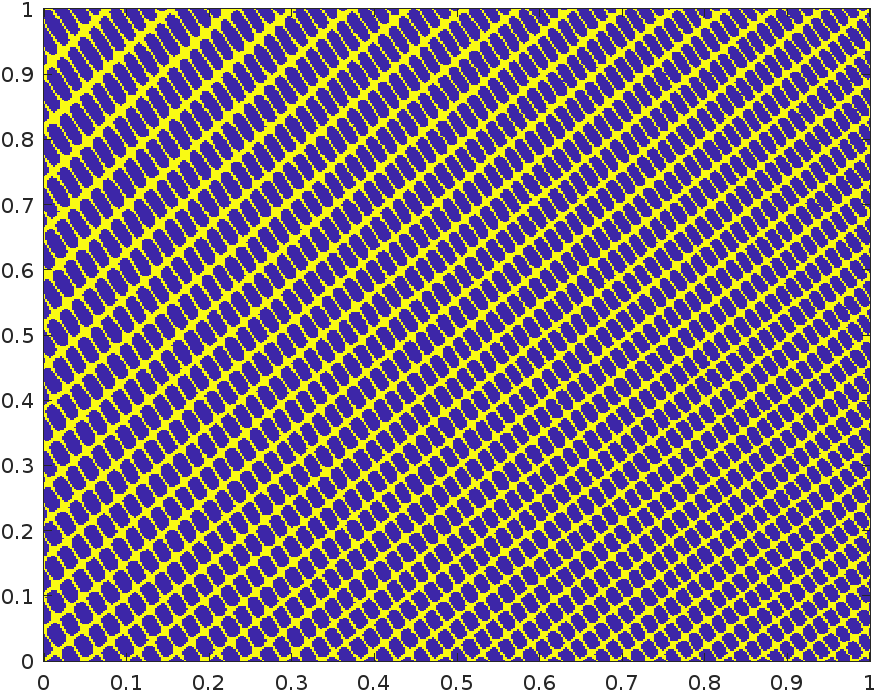}
    \quad
    \includegraphics[width=7cm]{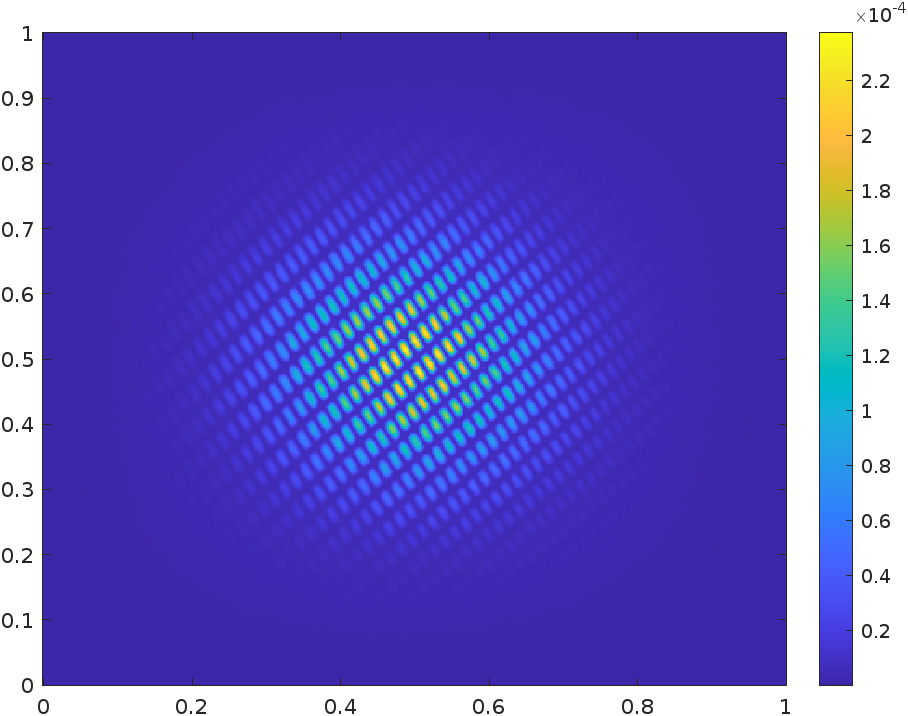}
    \caption{Left: Non-periodic field $\kappa$ in Example 2. Right: Reference solution at the final time $T$ when $\max \kappa = 10^5$ in Example 2.}
    \label{fig:Example_2_setting}
\end{figure}

\begin{table}
\centering
\caption{Estimates of $\inf \frac{\|W\|_{m_{22}}^2}{\|W\|_{a_{22}}^2}$ for different coarse mesh sizes $H$ and contrasts $\frac{\max{\kappa}}{\min{\kappa}}$ in Example 2}
\begin{tabular}{|c|c|c|c|c|c|c|c|c|}
\hline
\multirow{2}*{H} & \multicolumn{2}{|c|}{$\frac{\max{\kappa}}{\min{\kappa}}=10^4$} & \multicolumn{2}{|c|}{$\frac{\max{\kappa}}{\min{\kappa}}=10^5$} & \multicolumn{2}{|c|}{$\frac{\max{\kappa}}{\min{\kappa}}=10^6$} & \multicolumn{2}{|c|}{$\frac{\max{\kappa}}{\min{\kappa}}=10^7$}\\
\cline{2-9}
    & $V_{1,H}$ & $V_{2,H}$ & $V_{1,H}$ & $V_{2,H}$ & $V_{1,H}$ & $V_{2,H}$ & $V_{1,H}$ & $V_{2,H}$ \\
\hline

$1/10$ & $7.14 \times 10^{-7}$ & $9.80 \times 10^{-5}$ & 
$7.19 \times 10^{-8}$ & $9.80 \times 10^{-5}$ & 
$7.19 \times 10^{-9}$ & $9.80 \times 10^{-5}$ & 
$7.19 \times 10^{-10}$ & $9.81 \times 10^{-5}$ \\

$1/20$ & $1.66 \times 10^{-7}$ & $3.99 \times 10^{-5}$ & 
$1.67 \times 10^{-8} $ & $3.99 \times 10^{-5}$ & 
$1.67 \times 10^{-9} $ & $3.99 \times 10^{-5}$ & 
$1.67 \times 10^{-10}$ & $3.99 \times 10^{-5}$ \\

\hline
\end{tabular}
\label{tab:Example_2_condition_1}
\end{table}

\begin{table}
\centering
\caption{Estimates of $\inf \frac{\|W\|_{m_{22}}^2}{\|W\|_{a_{22}}^2+\|W\|_{c_{22}}^2}$ for different coarse mesh sizes $H$ and contrasts $\frac{\max{\kappa}}{\min{\kappa}}$ in Example 2}
\begin{tabular}{|c|c|c|c|c|c|c|c|c|}
\hline
\multirow{2}*{H} & \multicolumn{2}{|c|}{$\frac{\max{\kappa}}{\min{\kappa}}=10^4$} & \multicolumn{2}{|c|}{$\frac{\max{\kappa}}{\min{\kappa}}=10^5$} & \multicolumn{2}{|c|}{$\frac{\max{\kappa}}{\min{\kappa}}=10^6$} & \multicolumn{2}{|c|}{$\frac{\max{\kappa}}{\min{\kappa}}=10^7$}\\
\cline{2-9}
    & $V_{1,H}$ & $V_{2,H}$ & $V_{1,H}$ & $V_{2,H}$ & $V_{1,H}$ & $V_{2,H}$ & $V_{1,H}$ & $V_{2,H}$ \\
\hline

$1/10$ & $6.73 \times 10^{-7}$ & $1.08 \times 10^{-5}$ & 
$7.14 \times 10^{-8}$ & $1.08 \times 10^{-5}$ & 
$7.19 \times 10^{-9}$ & $1.08 \times 10^{-5}$ & 
$7.19 \times 10^{-10}$ & $1.08 \times 10^{-5}$ \\

$1/20$ & $1.64 \times 10^{-7}$ & $9.34 \times 10^{-6}$ & 
$1.67 \times 10^{-8} $ & $9.34 \times 10^{-6}$ & 
$1.67 \times 10^{-9} $ & $9.34 \times 10^{-6}$ & 
$1.67 \times 10^{-10}$ & $9.34 \times 10^{-6}$ \\

\hline
\end{tabular}
\label{tab:Example_2_condition_2}
\end{table}

\begin{figure}
  \centering
  \includegraphics[width=7cm]{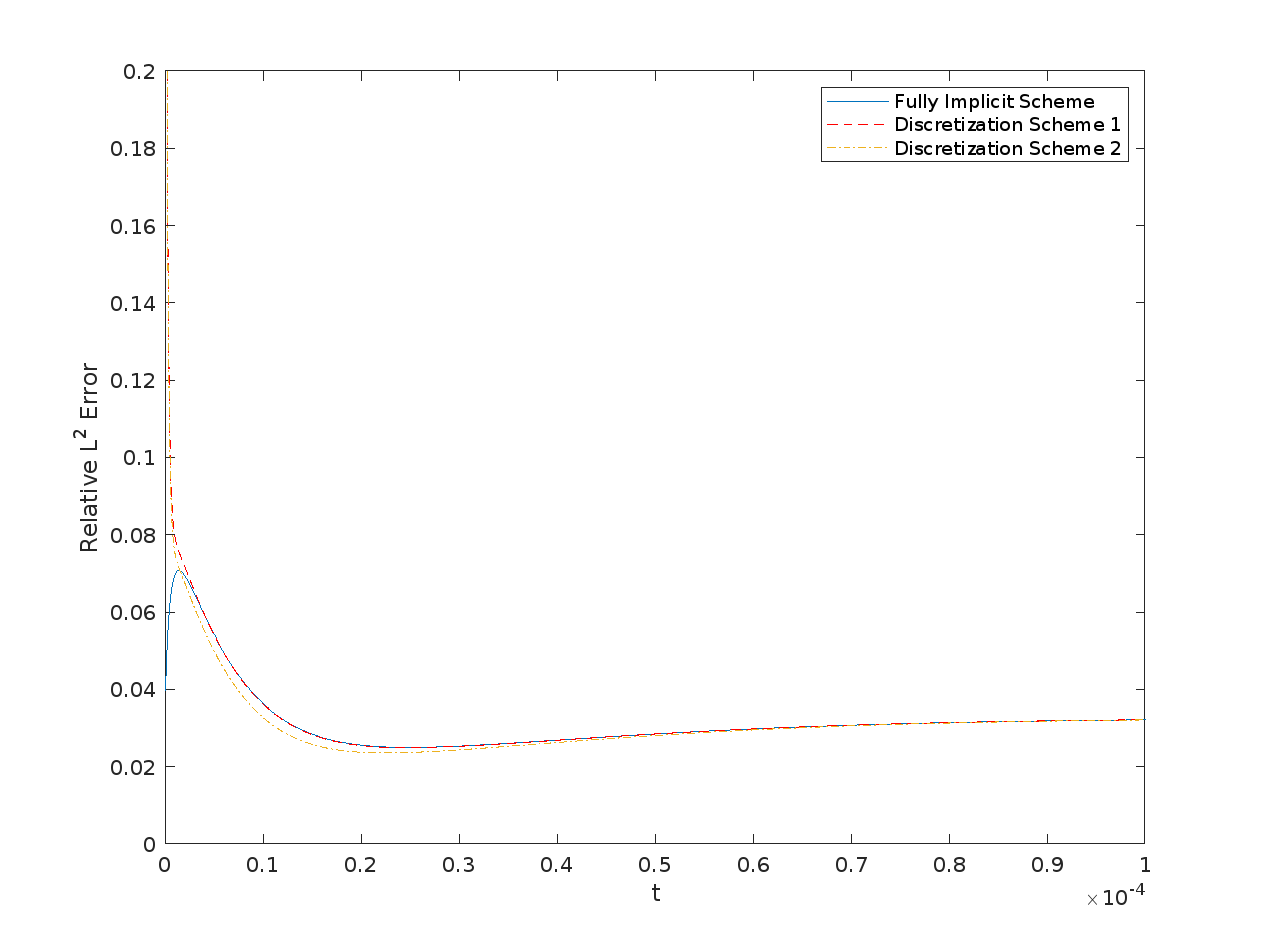}
  \quad
  \includegraphics[width=7cm]{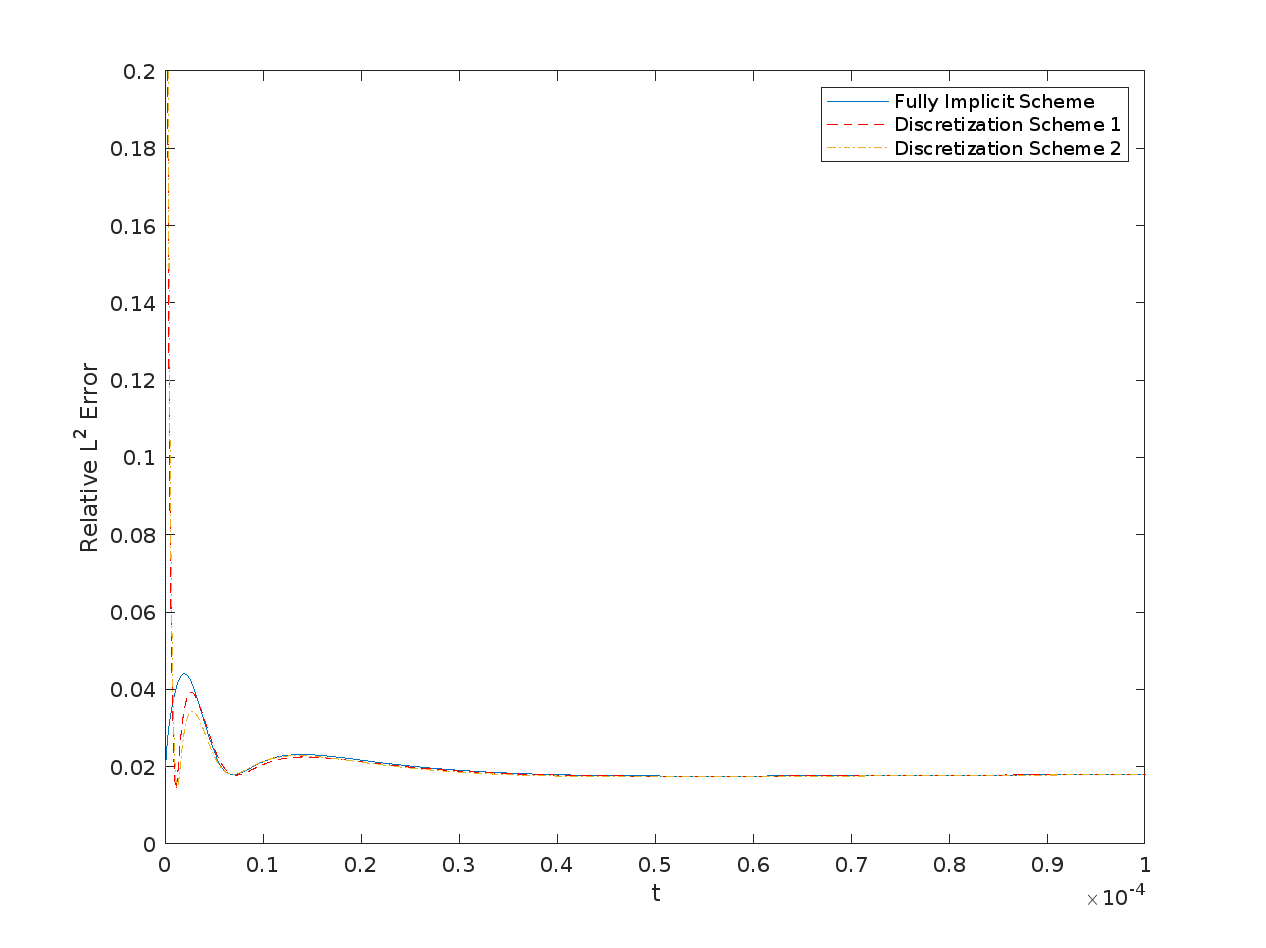}
  \quad
  \caption{Relative $L^2$ error for different schemes when $H=1/10$ and $l=5$ in Example 2. Left: $e^{(1)}(t)$. Right: $e^{(2)}(t)$.}
  \label{fig:Example_2_error_1}
\end{figure}

\begin{figure}
  \centering
  \includegraphics[width=7cm]{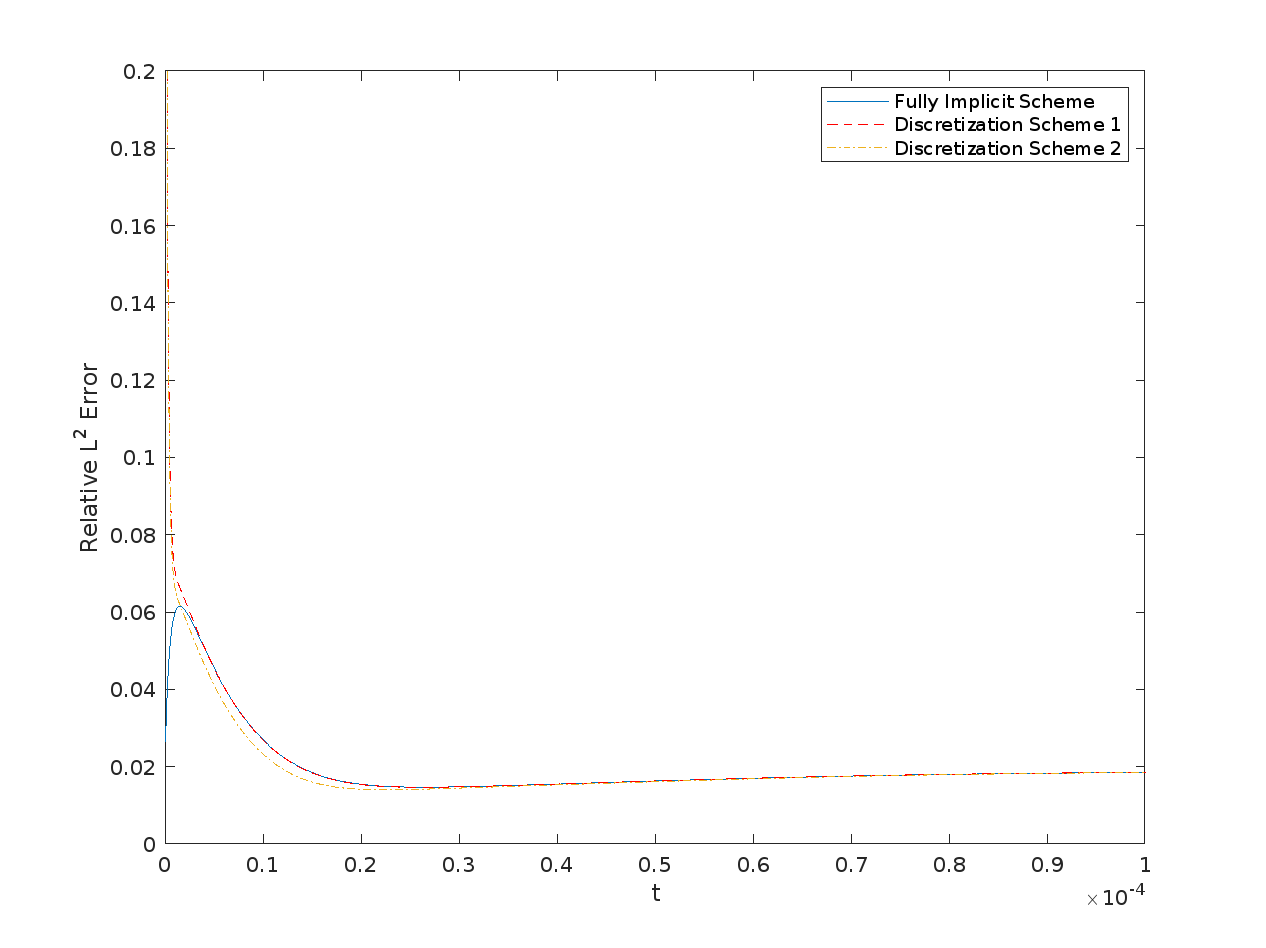}
  \quad
  \includegraphics[width=7cm]{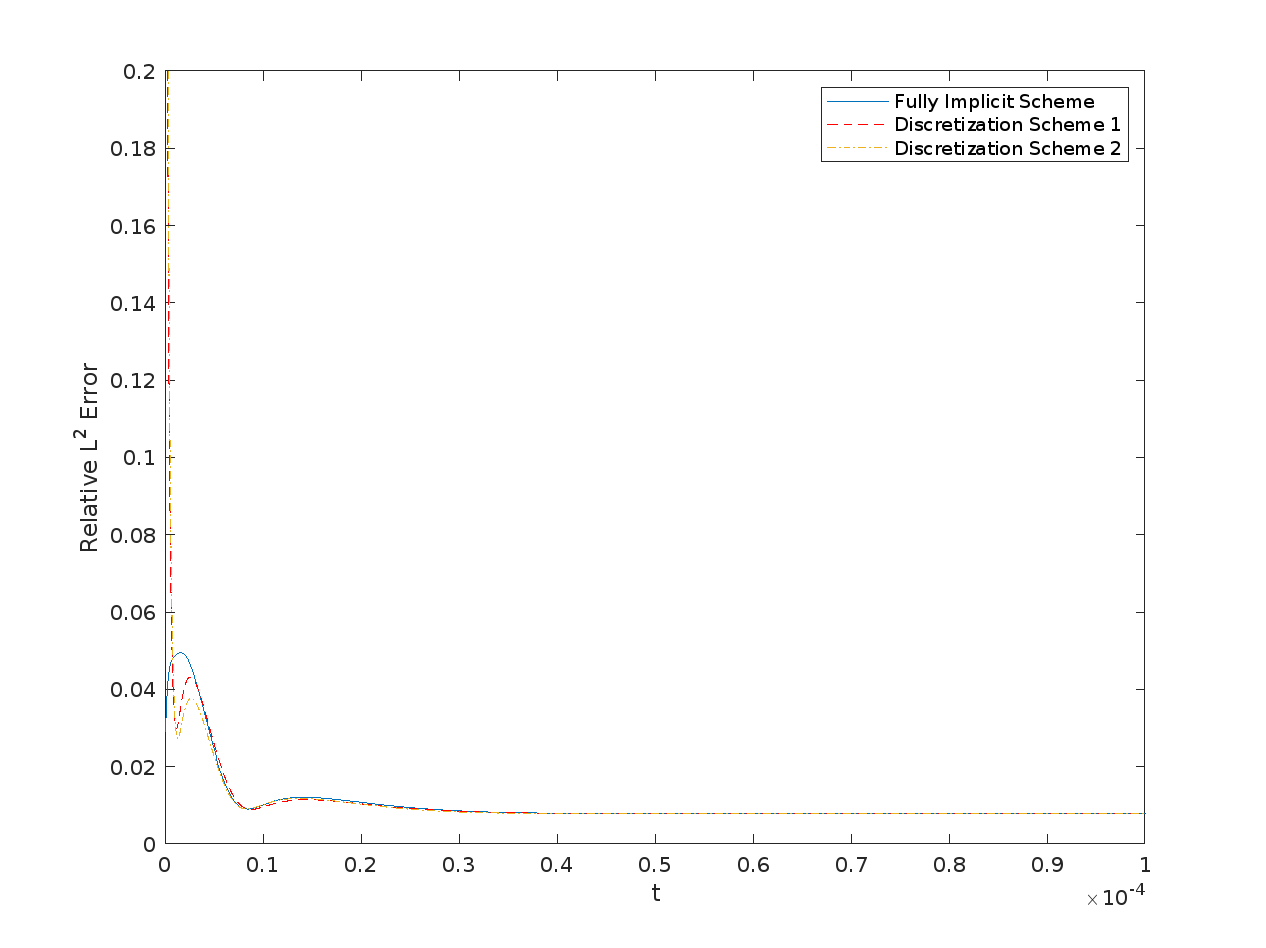}
  \quad
  \caption{Relative $L^2$ error for different schemes when $H=1/20$ and $l=6$ in Example 2. Left: $e^{(1)}(t)$. Right: $e^{(2)}(t)$.}
  \label{fig:Example_2_error_2}
\end{figure}

\subsection{Multiple-continuum cases}

In the following examples, we apply the optimized decomposition of the multicontinuum space to problems with multiple continua, which are defined in different patterns in each example to demonstrate the splitting results.

\subsubsection{Example 3}

We consider a layered field with three values as shown in Figure \ref{fig:Example_3_setting}, where the lowest value is $1$, the highest value is $10^7$, and the medium value is $10^2$. The corresponding three regions are denoted by $\Omega_1$, $\Omega_2$ and $\Omega_3$ respectively. We take the auxiliary functions $\psi_1$, $\psi_2$ and $\psi_3$ to be the characteristic functions in each of the three regions; namely, in a coarse block (RVE), we define $\psi_1(x) = \mathds{1}_{\Omega_1}(x)$, $\psi_2(x) = \mathds{1}_{\Omega_2}(x)$, and $\psi_3(x) = \mathds{1}_{\Omega_3}(x)$.
Following the discussions in Section \ref{sec:construction}, we present the results of eigenvalue problem (\ref{eq:Scheme_1_eig}) in Table \ref{tab:Example_3_eigen}. It can be seen that the second coordinates of the eigenvectors paired with small eigenvalues are zero, as expected. This indicates that the continua corresponding to the relatively low-value regions are considered to represent slow dynamics and are treated explicitly, while the continuum corresponding to the high-value regions is considered to represent fast dynamics and is treated implicitly.

\begin{figure}
    \centering
    \includegraphics[width=7cm]{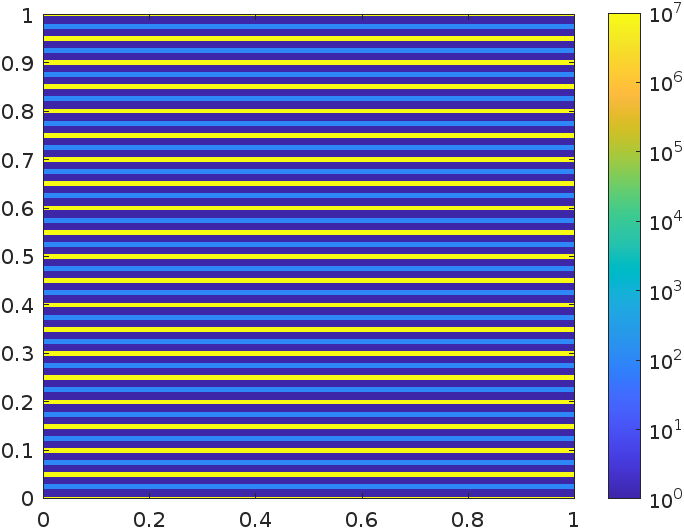}
    \quad
    \includegraphics[width=7cm]{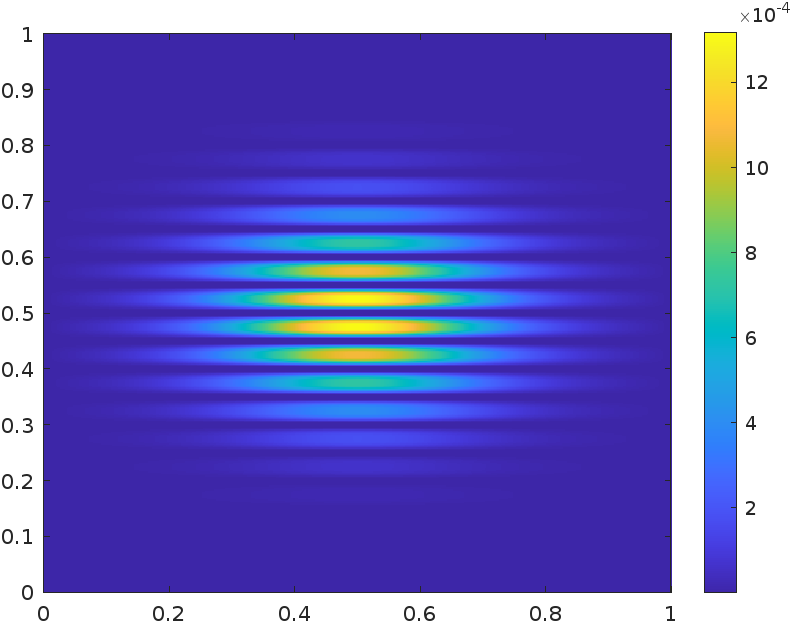}
    \caption{Left: Three-continuum field $\kappa$ in Example 3. Right: Reference solution at the final time $T$ in Example 3.}
    \label{fig:Example_3_setting}
\end{figure}

\begin{table}
    \centering
    \caption{Eigenvalues and corresponding eigenvectors of eigenvalue problems for different coarse mesh sizes $H$ in Example 3. Values below $5 \times 10^{-5}$ are listed as zero for simplification.}
    \begin{tabular}{|c|c|c|c|}
         \hline
         $H$ & $l$ & eigenvalues $\lambda_i$ &  eigenvectors $v_i$ \\
         \hline
         \multirow{3}*{1/10} & \multirow{3}*{5} & $1.0676 \times 10^{1}$ & $(1.1702, 0.0000, -0.0372)^T$ \\
         \cline{3-4}
         & & $8.7387 \times 10^{1}$ & $(0.1872, 0.0000, 1.9044)^T$ \\
         \cline{3-4}
         & & $7.3090 \times 10^{6}$ & $(0.1811, 1.9097, 0.1653)^T$ \\
         \hline
         \multirow{3}*{1/20} & \multirow{3}*{6} & $5.3479 \times 10^{0}$ & $(1.1718, 0.0000, -0.0197)^T$ \\
         \cline{3-4}
         & & $7.7783 \times 10^{1}$ & $(0.1765, 0.0000, 1.9047)^T$ \\
         \cline{3-4}
         & & $7.2902 \times 10^{6}$ & $(0.1811, 1.9097, 0.1653)^T$ \\
         \hline
    \end{tabular}   
    \label{tab:Example_3_eigen}
\end{table}

We choose $T = 6\times 10^{-4}$ and $\Delta t = 5\times 10^{-7}$. The almost-steady state at $T$ is presented in Figure \ref{fig:Example_3_setting}. We depict the relative $L^2$ errors for different schemes with respect to the reference solution in Figures \ref{fig:Example_3_error_1} and \ref{fig:Example_3_error_2}, where the divergent fully explicit scheme is not included. The figures on the left, middle and right describe the relative errors in the low-, high- and medium-value regions respectively. It can be noted that the partially explicit schemes achieve accuracy comparable to the fully implicit scheme, while requiring less computational effort. Additionally, the errors can be reduced by choosing smaller $H$.

\begin{figure}
  \centering
  \includegraphics[width=5cm]{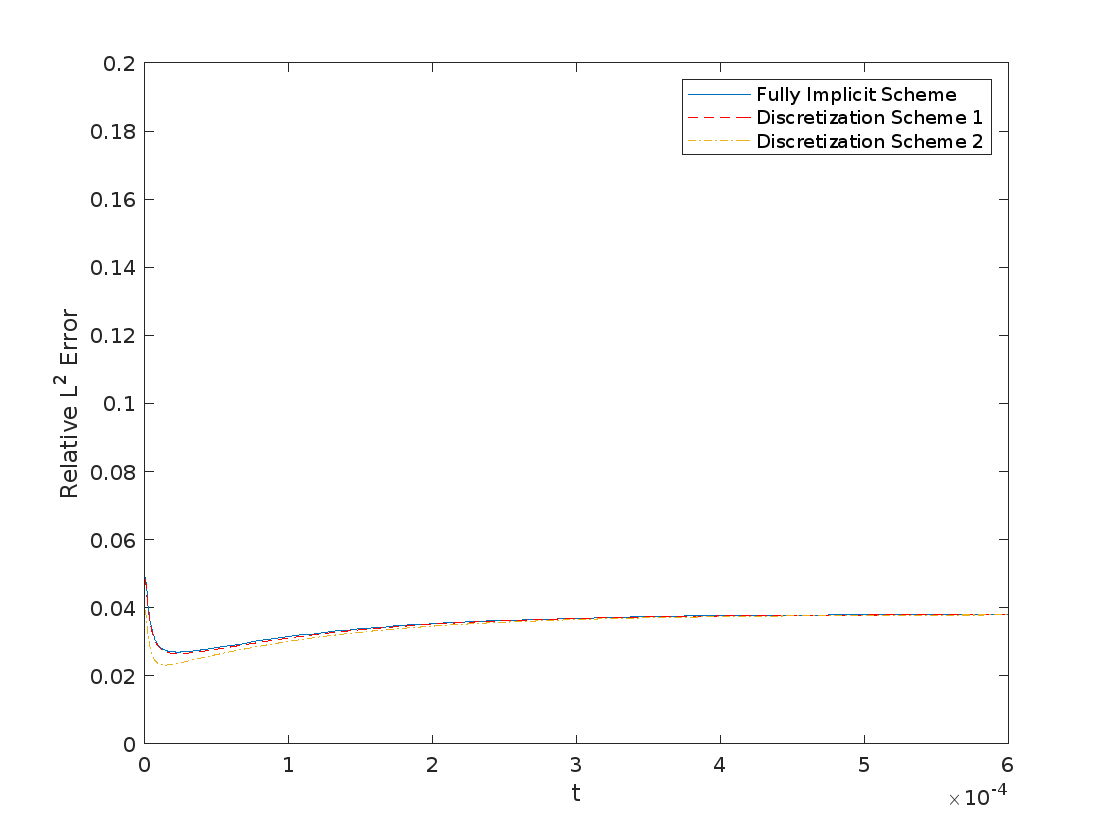}
  \quad
  \includegraphics[width=5cm]{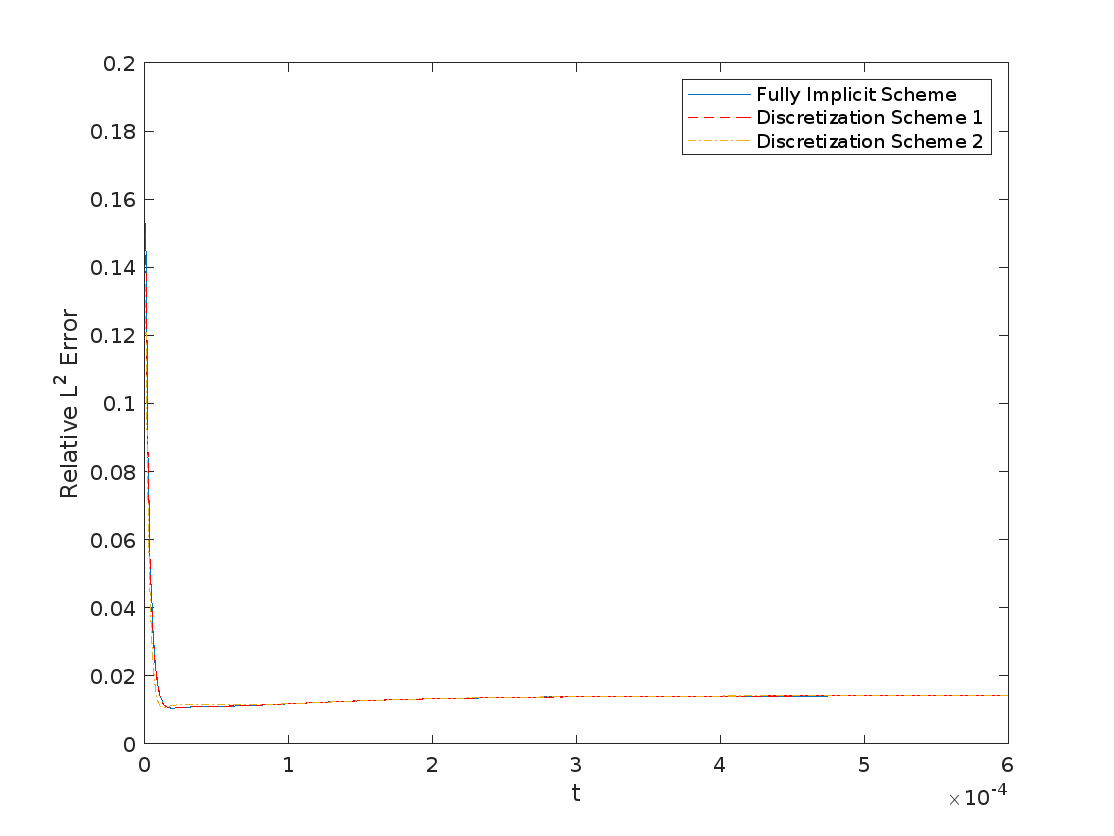}
  \quad
  \includegraphics[width=5cm]{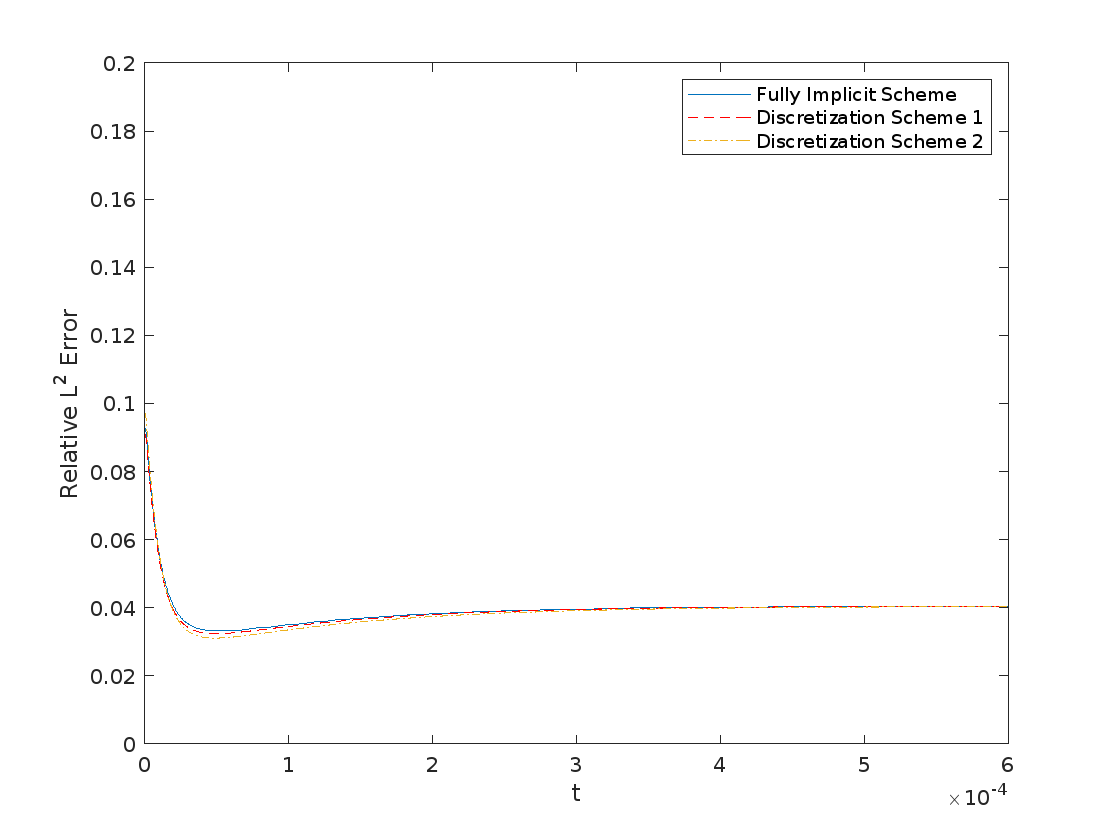}
  \caption{Relative $L^2$ error for different schemes when $H=1/10$ and $l=5$ in Example 3. Left: $e^{(1)}(t)$. Middle: $e^{(2)}(t)$. Right: $e^{(3)}(t)$.}
  \label{fig:Example_3_error_1}
\end{figure}

\begin{figure}
  \centering
  \includegraphics[width=5cm]{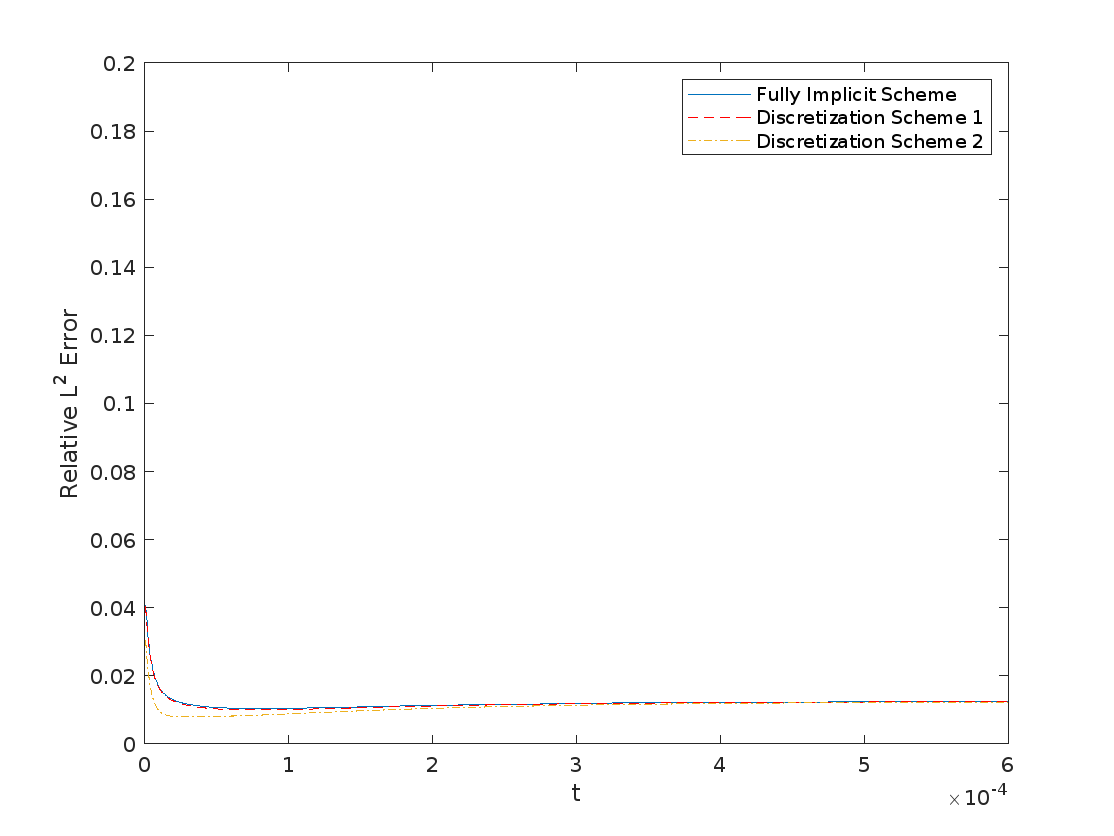}
  \quad
  \includegraphics[width=5cm]{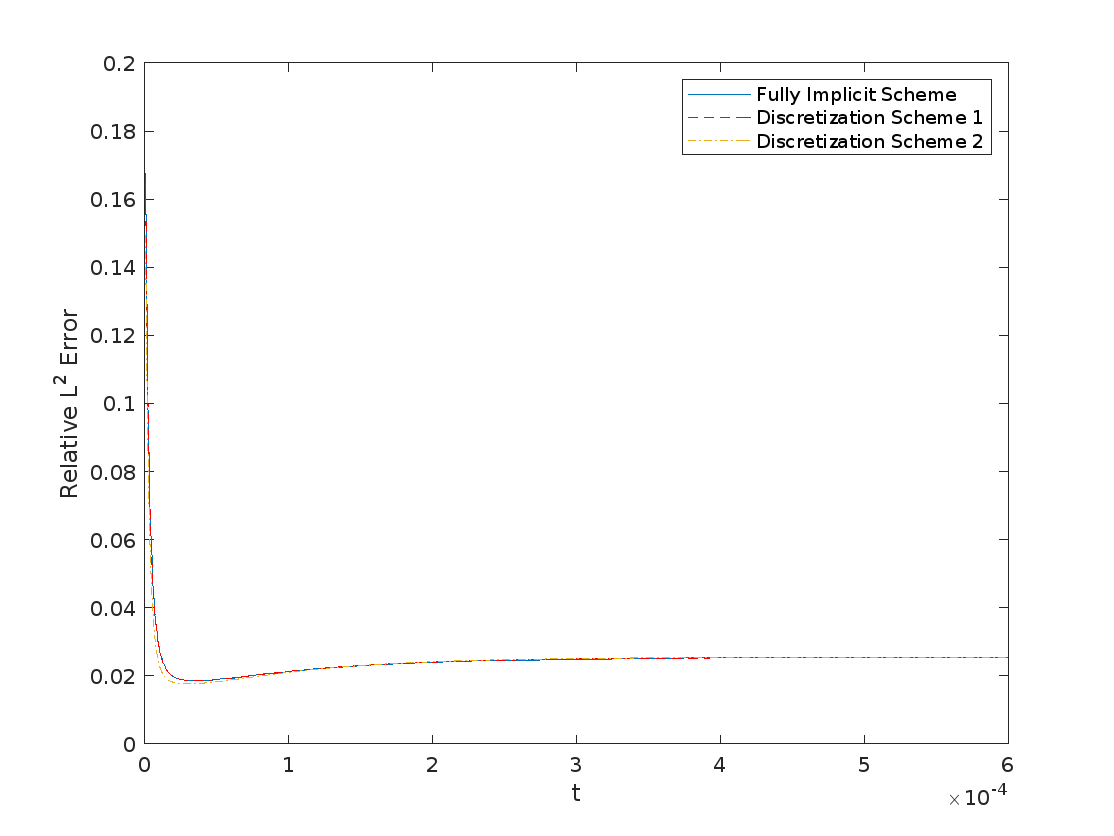}
  \quad
  \includegraphics[width=5cm]{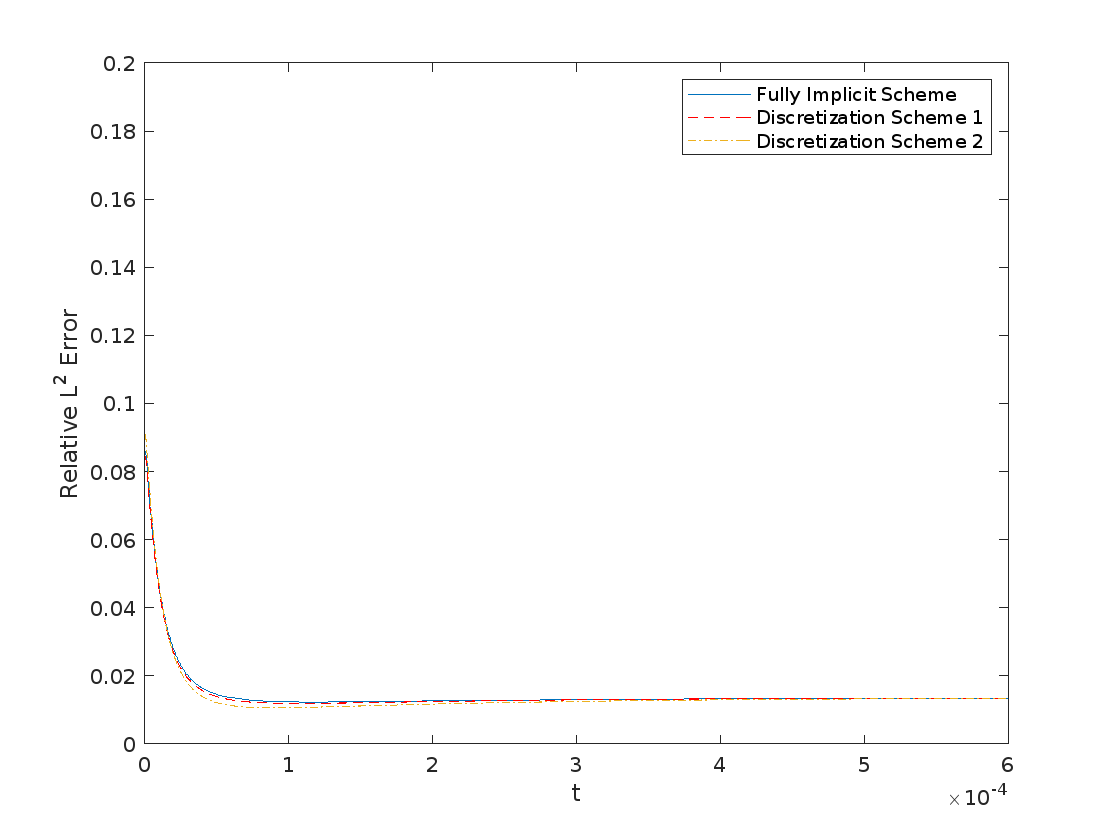}
  \caption{Relative $L^2$ error for different schemes when $H=1/20$ and $l=6$ in Example 3. Left: $e^{(1)}(t)$. Middle: $e^{(2)}(t)$. Right: $e^{(3)}(t)$.}
  \label{fig:Example_3_error_2}
\end{figure}

\subsubsection{Example 4}

In Example 4, we again consider the three-continuum case, but with a different permeability field as shown in Figure \ref{fig:Example_4_setting} and different definitions of continua (or auxiliary functions $\psi$'s). The lowest value of the field is $1$, the highest value is $10^7$, and the medium value is $10^2$, with the corresponding regions denoted by $\Omega_1$, $\Omega_2$ and $\Omega_3$. 
The auxiliary functions $\psi_1$, $\psi_2$ and $\psi_3$ are respectively chosen as the characteristic functions in low-high-value regions, low-medium-value regions and medium-high-value regions; that is, $\psi_1(x) = \mathds{1}_{\Omega_1 \cup \Omega_2}(x)$, $\psi_2(x) = \mathds{1}_{\Omega_1 \cup \Omega_3}(x)$, and $\psi_3(x) = \mathds{1}_{\Omega_2 \cup \Omega_3}(x)$.
We can similarly obtain the results of eigenvalue problems in Table \ref{tab:Example_4_eigen}. It can be observed that by sufficiently mixing and taking linear combinations of the continua, our algorithm effectively separates the slow modes corresponding to the two low eigenvalues from the overall dynamics. The slow modes are processed explicitly to reduce computational cost, while the remaining fast mode is process implicitly to ensure stability.

\begin{figure}
    \centering
    \includegraphics[width=7cm]{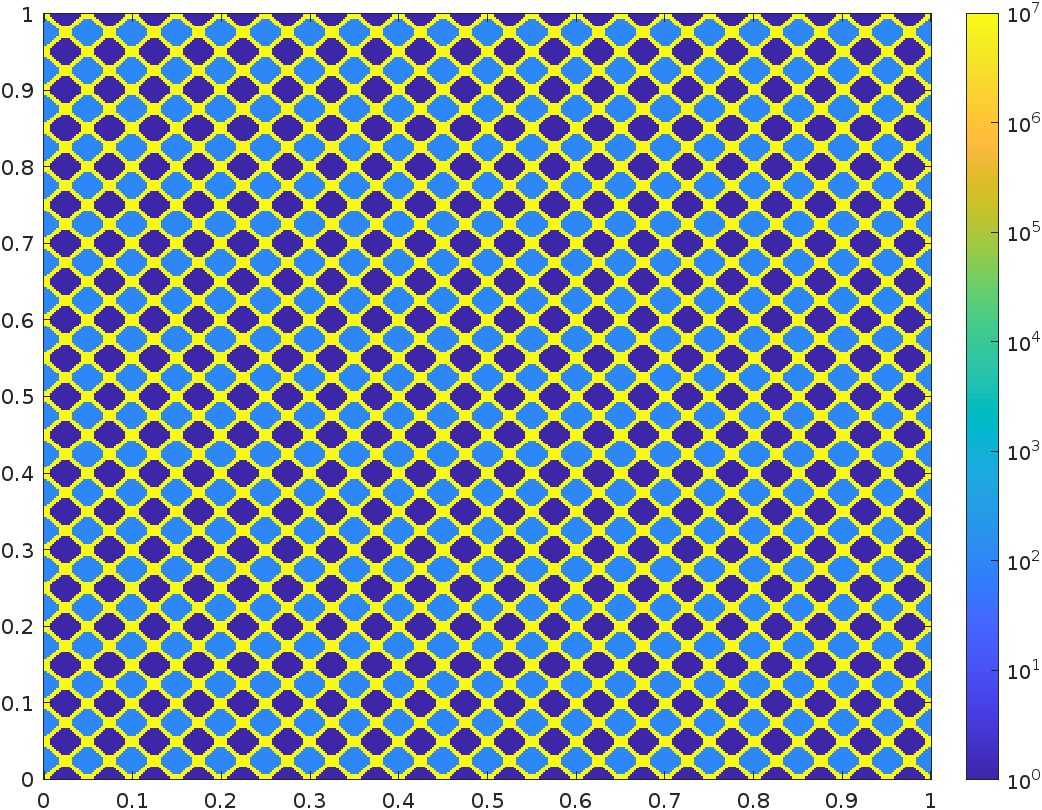}
    \quad
    \includegraphics[width=7cm]{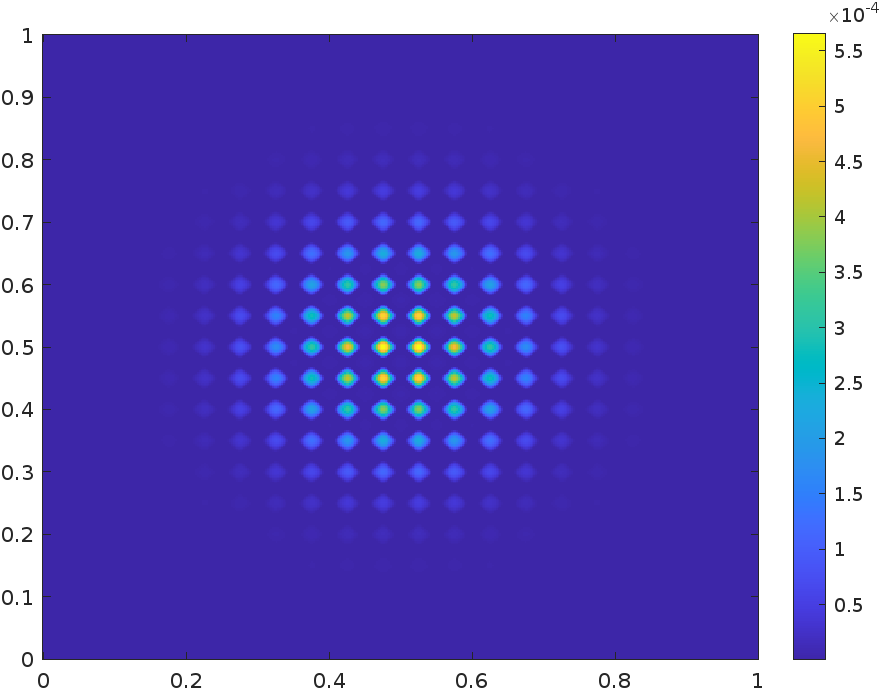}
    \caption{Left: Three-continuum field $\kappa$ in Example 4. Right: Reference solution at the final time $T$ in Example 4.}
    \label{fig:Example_4_setting}
\end{figure}

\begin{table}
    \centering
    \caption{Eigenvalues and corresponding eigenvectors of eigenvalue problems for different coarse mesh sizes $H$ in Example 4. Values below $5 \times 10^{-5}$ are listed as zero for simplification.}
    \begin{tabular}{|c|c|c|c|}
         \hline
         $H$ & $l$ & eigenvalues $\lambda_i$ &  eigenvectors $v_i$ \\
         \hline
         \multirow{3}*{1/10} & \multirow{3}*{5} & $7.5039 \times 10^{2}$ & $(-0.4058, 0.2058, 0.6181)^T$ \\
         \cline{3-4}
         & & $4.8925 \times 10^{3}$ & $(0.6170, 0.9933, 0.4051)^T$ \\
         \cline{3-4}
         & & $3.5383 \times 10^{6}$ & $(0.8840, 0.4151, 0.8835)^T$ \\
         \hline
         \multirow{3}*{1/20} & \multirow{3}*{6} & $3.5031 \times 10^{2}$ & $(-0.4389, 0.1513, 0.5949)^T$ \\
         \cline{3-4}
         & & $2.1999 \times 10^{3}$ & $(0.5944, 1.0032, 0.4386)^T$ \\
         \cline{3-4}
         & & $3.5261 \times 10^{6}$ & $(0.8836, 0.4147, 0.8834)^T$ \\
         \hline
    \end{tabular}   
    \label{tab:Example_4_eigen}
\end{table}

Since the critical eigenvalue in this example is higher than in Example 3 due to the properties of the coefficient field, we choose a finer time step, such as $\Delta t = 5 \times 10^{-8}$. We set the final time to $T = 1.5\times 10^{-4}$, at which point the reference solution, having reached an almost-steady state, is plotted in Figure \ref{fig:Example_4_setting}. 
Then we solve the parabolic equation using the multicontinuum fully implicit scheme and the multicontinuum splitting schemes. It is important to note that the macroscopic solutions obtained by our proposed schemes are linear combinations of the multicontinuum variables induced by $\psi_i$'s. Therefore, we need to project the macroscopic solutions back before obtaining the multiscale solutions in continuum $i$ and calculating the errors. 
The relative $L^2$ errors for different schemes are presented in Figures \ref{fig:Example_4_error_1} and \ref{fig:Example_4_error_2}. Comparable patterns to those in Example 3 are observed.

\begin{figure}
  \centering
  \includegraphics[width=5cm]{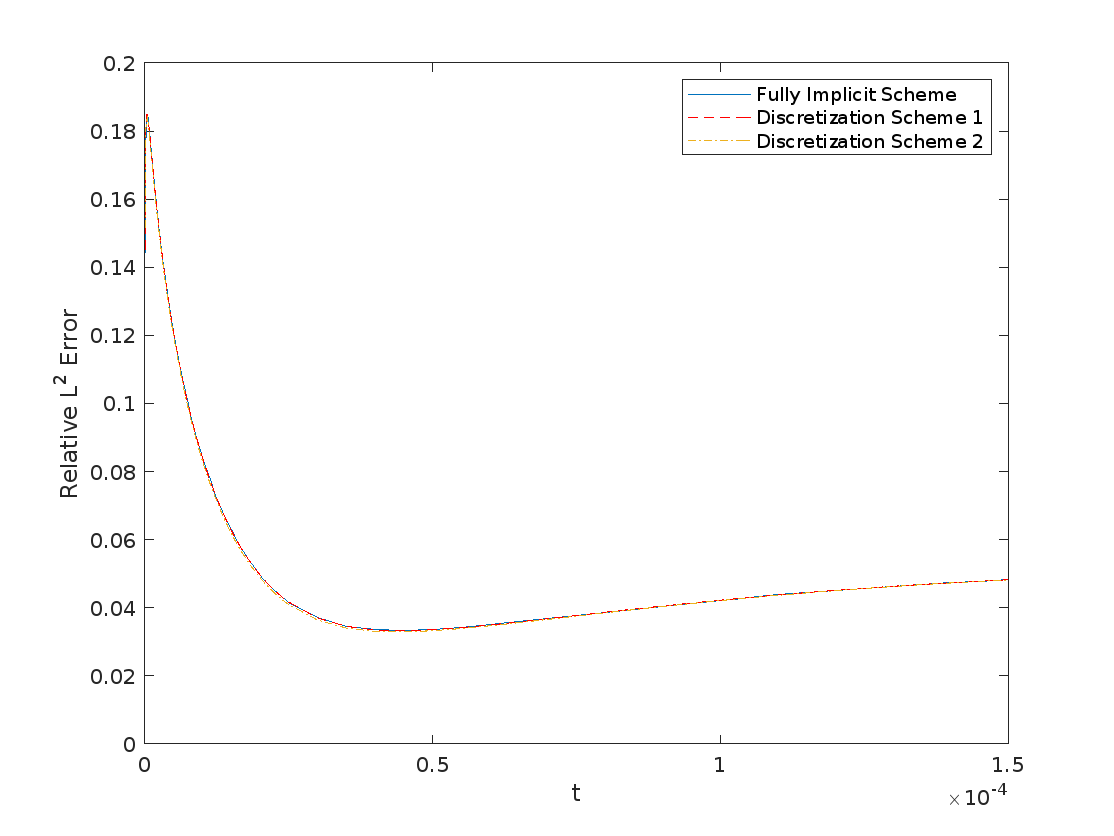}
  \quad
  \includegraphics[width=5cm]{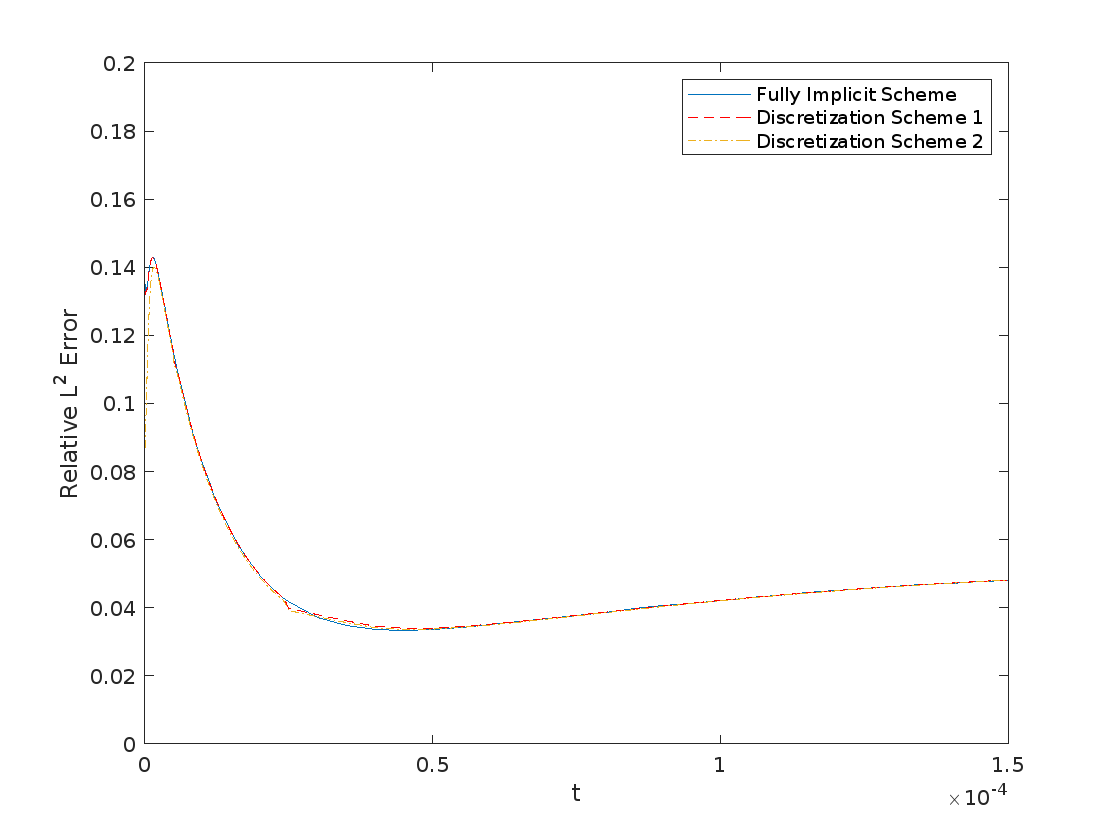}
  \quad
  \includegraphics[width=5cm]{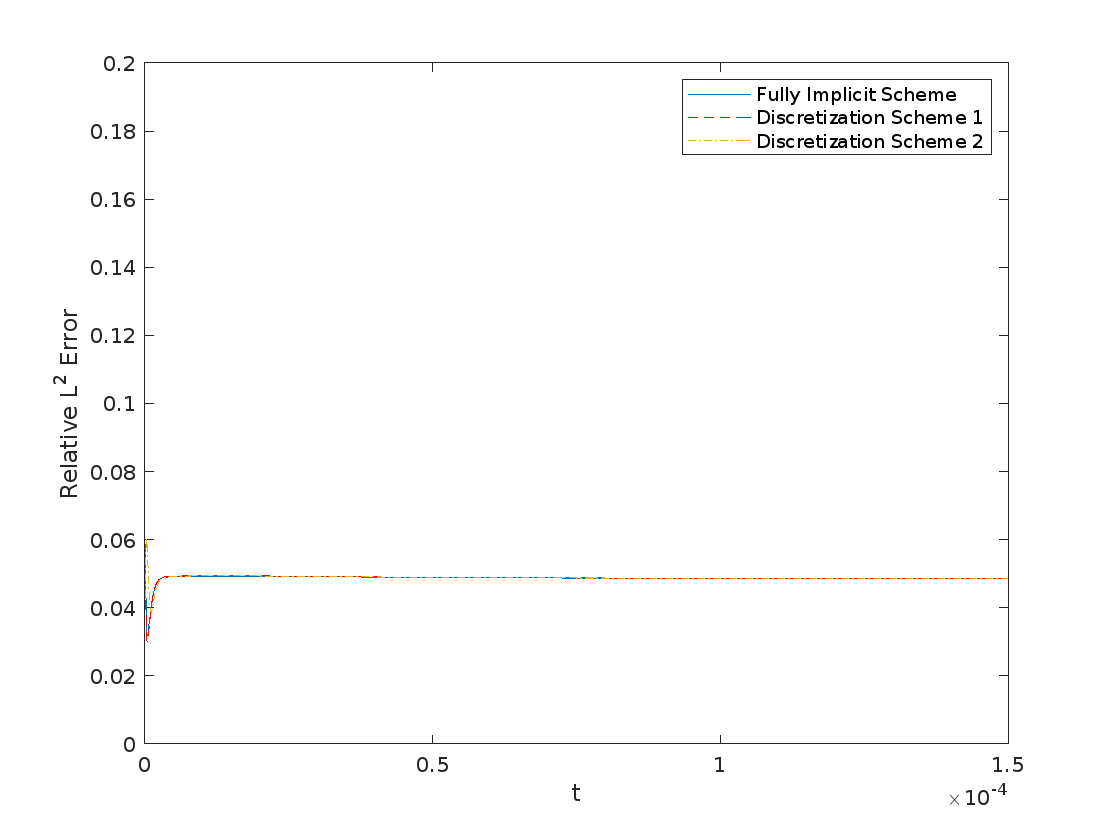}
  \caption{Relative $L^2$ error for different schemes when $H=1/10$ and $l=5$ in Example 4. Left: $e^{(1)}(t)$. Middle: $e^{(2)}(t)$. Right: $e^{(3)}(t)$.}
  \label{fig:Example_4_error_1}
\end{figure}

\begin{figure}
  \centering
  \includegraphics[width=5cm]{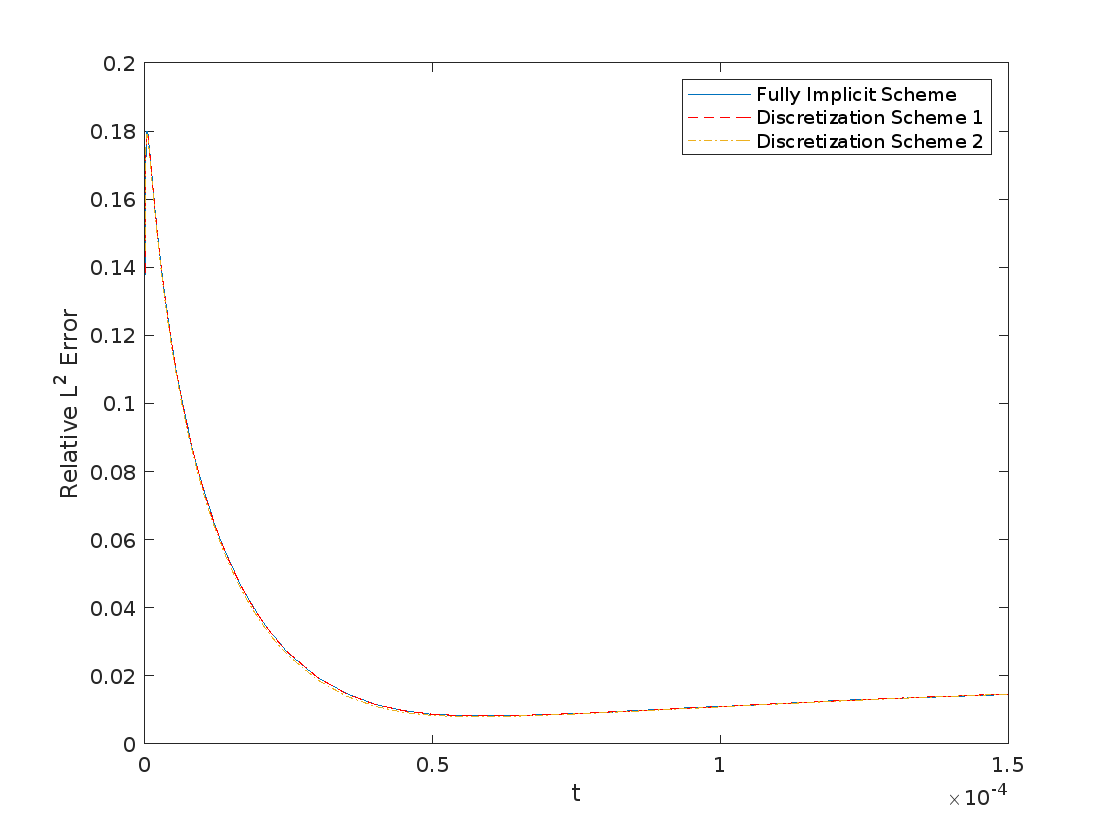}
  \quad
  \includegraphics[width=5cm]{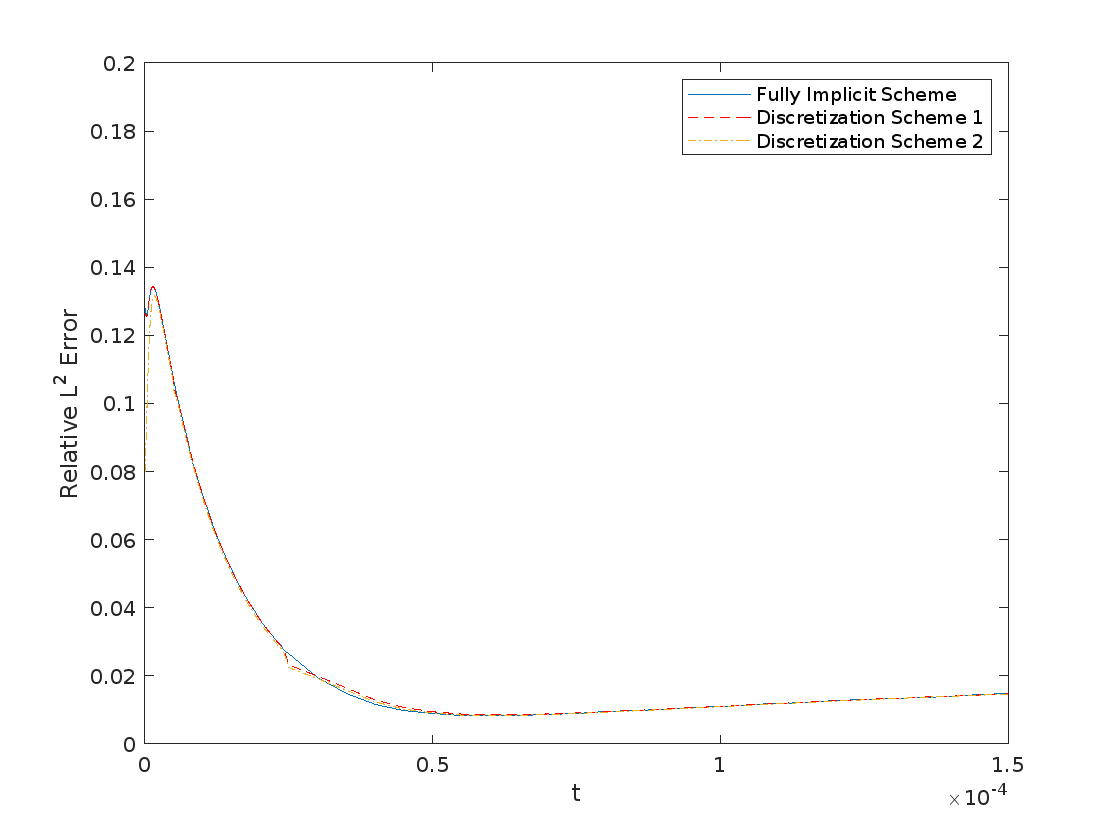}
  \quad
  \includegraphics[width=5cm]{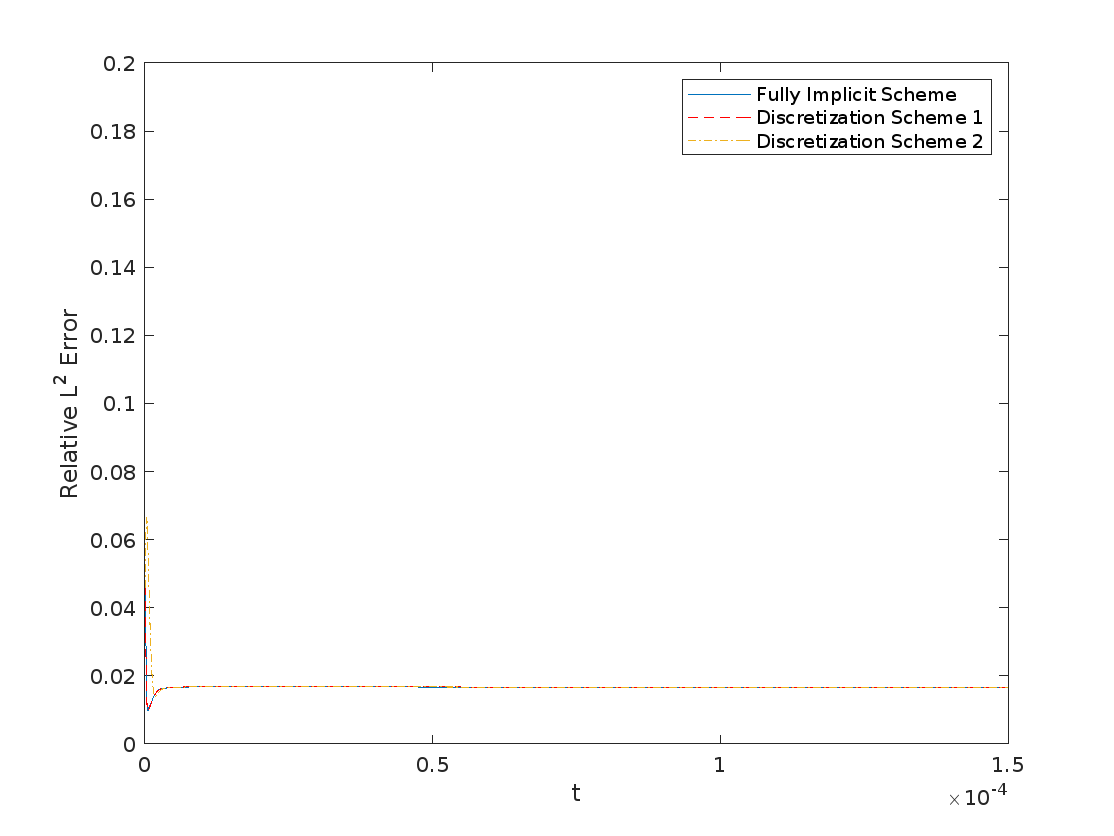}
  \caption{Relative $L^2$ error for different schemes when $H=1/20$ and $l=6$ in Example 4. Left: $e^{(1)}(t)$. Middle: $e^{(2)}(t)$. Right: $e^{(3)}(t)$.}
  \label{fig:Example_4_error_2}
\end{figure}

\subsubsection{Example 5}

In this example, we consider a more general case where not all the auxiliary functions are characteristic functions. 
We consider the same layered field $\kappa$ as in Example 1, with $\max{\kappa} = 10^5$. We introduce four auxiliary functions representing constants and linear functions within the regions of low and high values. Specifically, in a coarse block (RVE), we define $\psi_1(x) = \mathds{1}_{\Omega_1}(x)$, $\psi_2(x) = \mathds{1}_{\Omega_2}(x)$, $\psi_3(x) = (x_1-c_1) \mathds{1}_{\Omega_1}(x) $, and $\psi_4(x) = (x_1-c_1) \mathds{1}_{\Omega_2}(x)$, where $\Omega_1$ and $\Omega_2$ are respectively the low-value and high-value regions, and $c_1$ is the $x_1$-coordinate of the leftmost point of the coarse block. 
The results of the eigenvalue problems are presented in Table \ref{tab:Example_5_eigen}. We observe that the slow and fast dynamics are separated as expected. The continua associated with the characteristic function and its linear component in the low-value regions represent slow modes and are classified into the explicit component, while those associated with the high-value regions represent fast modes and are classified into the implicit component.

\begin{table}
    \centering
    \caption{Eigenvalues and corresponding eigenvectors of eigenvalue problems for different coarse mesh sizes $H$ in Example 5. Values below $5 \times 10^{-5}$ are listed as zero for simplification.}
    \begin{tabular}{|c|c|c|c|}
         \hline
         $H$ & $l$ & eigenvalues $\lambda_i$ &  eigenvectors $v_i$ \\
         \hline
         \multirow{4}*{1/10} & \multirow{4}*{5} & $1.6617 \times 10^{1}$ & $(0.2985, 0.0000, -0.2195, 0.0000)^T$ \\
         \cline{3-4}
         & & $2.7939 \times 10^{2}$ & $(0.9732, 0.0000, 1.1321, 0.0000)^T$ \\
         \cline{3-4}
         & & $5.3862 \times 10^{5}$ & $(0.0695, 0.4060, -0.0172, -0.3585)^T$ \\
         \cline{3-4}
         & & $8.7003 \times 10^{6}$ & $(0.2865, 1.6745, 0.3075, 1.8599)^T$ \\
         \hline
         \multirow{4}*{1/20} & \multirow{4}*{6} & $1.0879 \times 10^{1}$ & $(0.2714, 0.0000, -0.2271, 0.0000)^T$ \\
         \cline{3-4}
         & & $1.9384 \times 10^{2}$ & $(0.9811, 0.0000, 1.1190, 0.0000)^T$ \\
         \cline{3-4}
         & & $6.7157 \times 10^{5}$ & $(0.0751, 0.4389, -0.0251, -0.3626)^T$ \\
         \cline{3-4}
         & & $7.5412 \times 10^{6}$ & $(0.2850, 1.6662, 0.3114, 1.8773)^T$ \\
         \hline
         \multirow{4}*{1/40} & \multirow{4}*{8} & $9.5558 \times 10^{0}$ & $(0.2574, 0.0000, -0.2273, 0.0000)^T$ \\
         \cline{3-4}
         & & $1.7885 \times 10^{2}$ & $(0.9849, 0.0000, 1.1115, 0.0000)^T$ \\
         \cline{3-4}
         & & $8.6817 \times 10^{5}$ & $(0.0828, 0.4842, -0.0208, -0.3324)^T$ \\
         \cline{3-4}
         & & $5.5922 \times 10^{6}$ & $(0.2829, 1.6536, 0.3132, 1.8927)^T$ \\
         \hline
    \end{tabular}   
    \label{tab:Example_5_eigen}
\end{table}

As in Example 1, we set the final time to $T=1.5 \times 10^{-4}$ and the step to $\Delta t=10^{-7}$. We again need to project the macroscopic solutions before computing the errors. The relative errors of different schemes are depicted in Figures \ref{fig:Example_5_error_1}, \ref{fig:Example_5_error_2} and \ref{fig:Example_5_error_3}. It can be observed that the initial errors of the proposed method are slightly smaller than those in Example 1, but overall they are similar and remain minor.

\begin{figure}
  \centering
  \includegraphics[width=7cm]{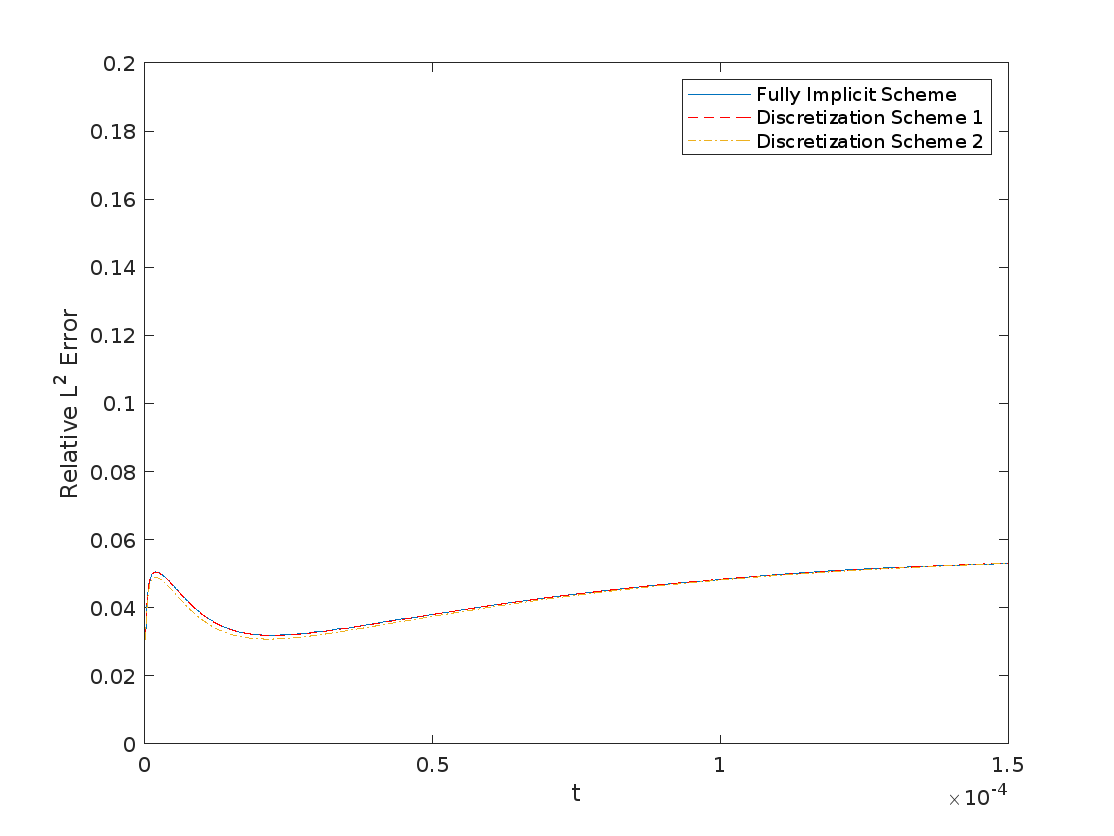}
  \quad
  \includegraphics[width=7cm]{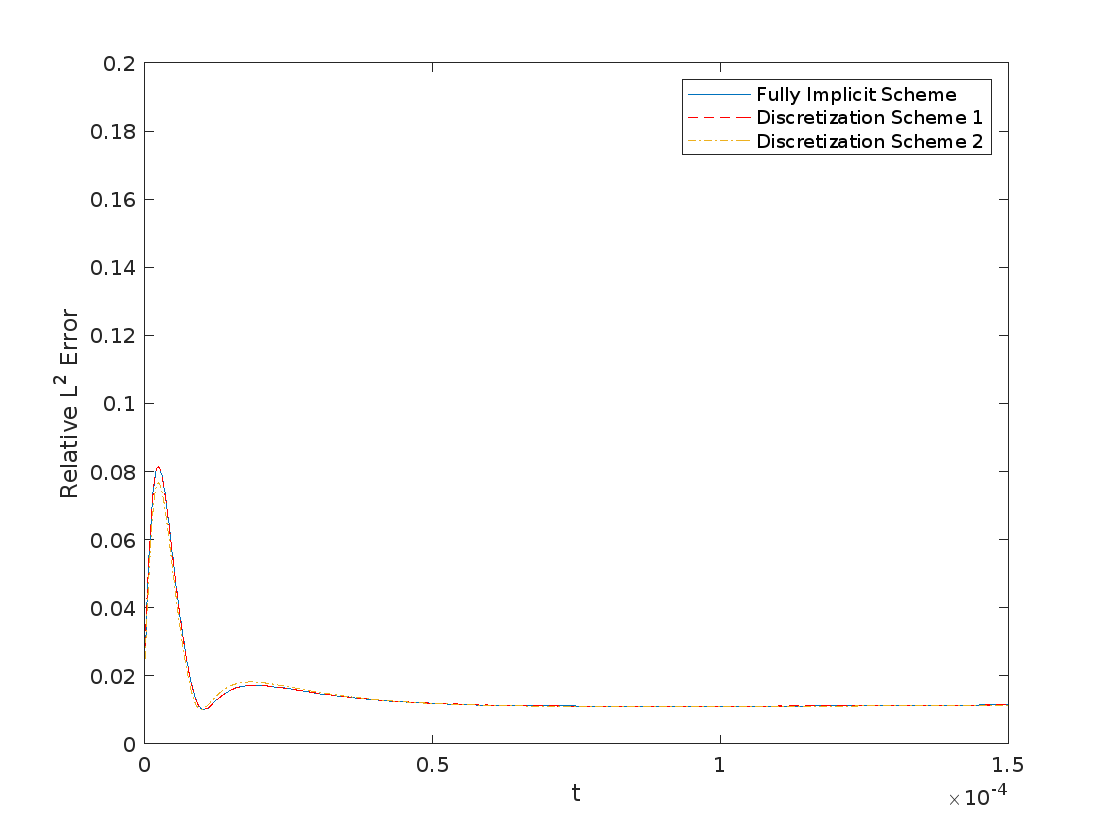}
  \quad
  \caption{Relative $L^2$ error for different schemes when $H=1/10$ and $l=5$ in Example 5. Left: $e^{(1)}(t)$. Right: $e^{(2)}(t)$.}
  \label{fig:Example_5_error_1}
\end{figure}

\begin{figure}
  \centering
  \includegraphics[width=7cm]{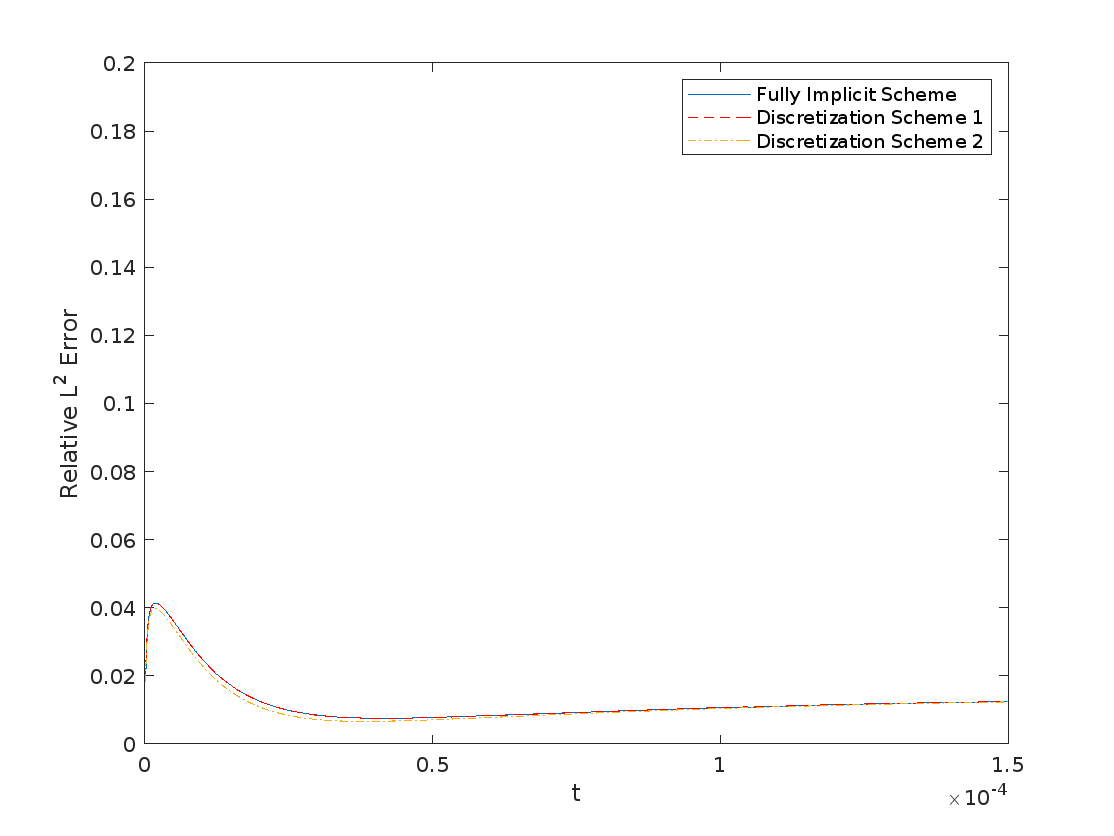}
  \quad
  \includegraphics[width=7cm]{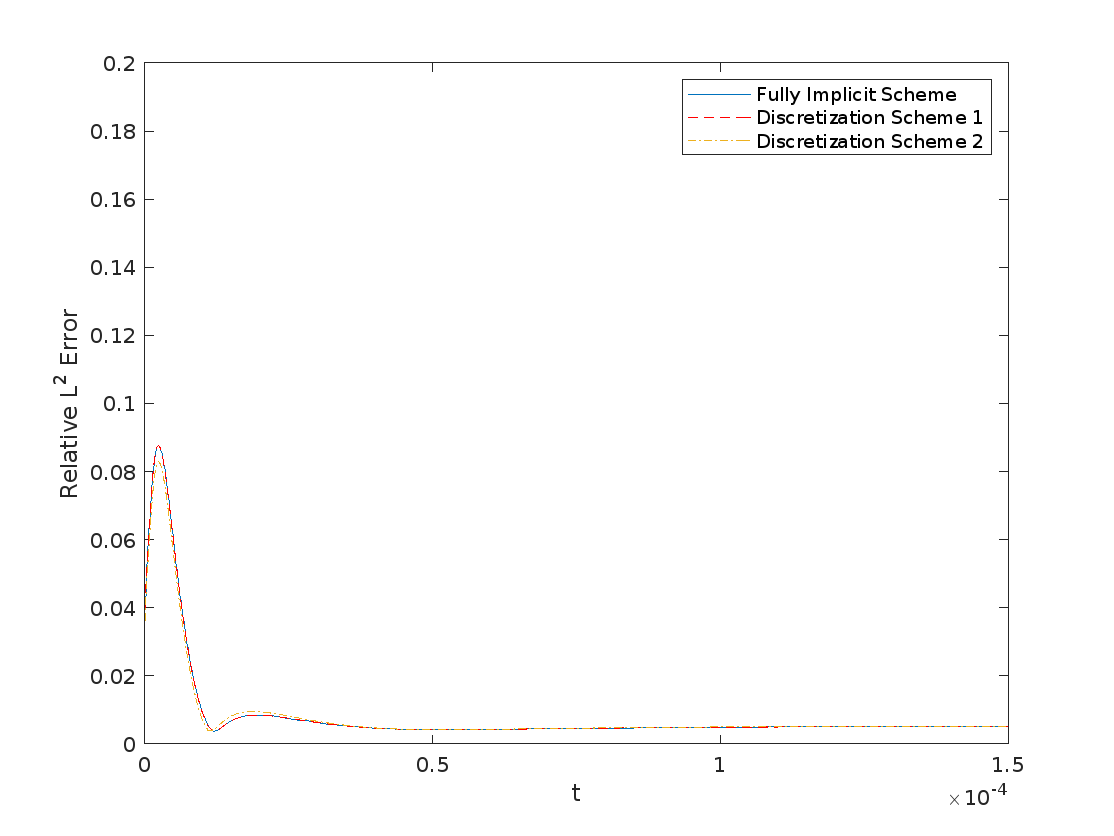}
  \quad
  \caption{Relative $L^2$ error for different schemes when $H=1/20$ and $l=6$ in Example 5. Left: $e^{(1)}(t)$. Right: $e^{(2)}(t)$.}
  \label{fig:Example_5_error_2}
\end{figure}

\begin{figure}
  \centering
  \includegraphics[width=7cm]{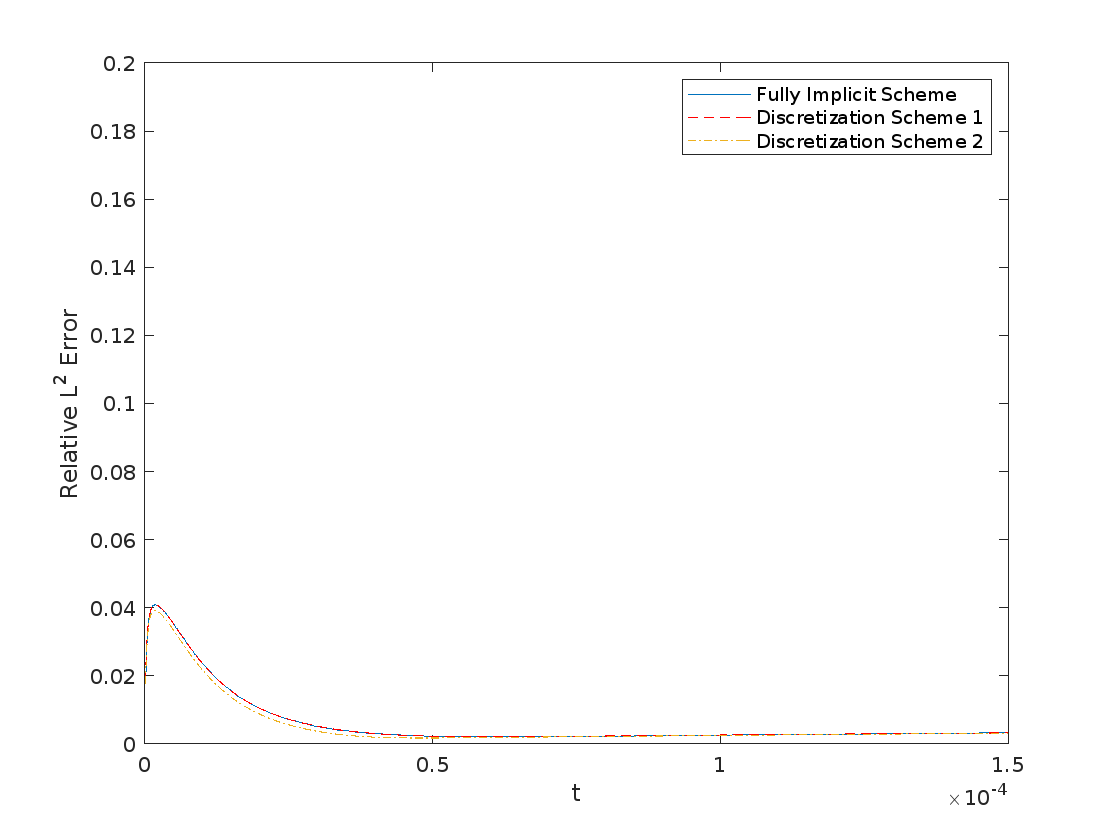}
  \quad
  \includegraphics[width=7cm]{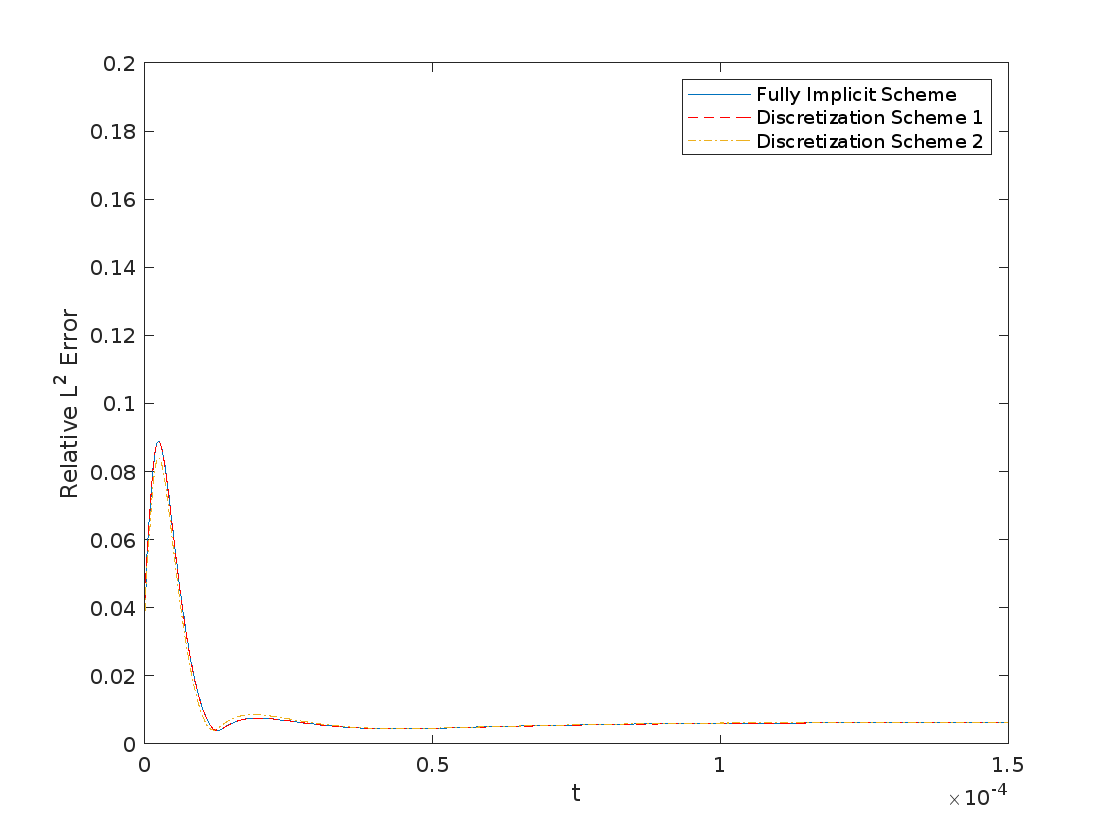}
  \quad
  \caption{Relative $L^2$ error for different schemes when $H=1/40$ and $l=8$ in Example 5. Left: $e^{(1)}(t)$. Right: $e^{(2)}(t)$.}
  \label{fig:Example_5_error_3}
\end{figure}

\section{Conclusions}
\label{sec:conclusions}

In this paper, we have developed multicontinuum splitting schemes for parabolic equations with high-contrast coefficients. By introducing macroscopic variables and utilizing multicontinuum expansion, we decompose the solution space into two components. This decomposition aims to separate the dynamics at different speeds (or the effects of contrast). By treating the component containing fast dynamics (or dependent on the contrast) implicitly and the one containing slow dynamics (or independent of the contrast) explicitly, we design partially explicit time discretization schemes to reduce computational cost. The stability conditions on the time step size are contrast-independent, provided the continua are chosen appropriately. Furthermore, we consider a more general case where the space decomposition is not predefined, and we discuss possible methods to achieve an optimized decomposition, which can further relax the stability conditions and improve the efficiency. The min-max optimization problem can be simplified into a generalized eigenvalue problem under certain assumptions.
Numerical results for various coefficient fields and different numbers of continua are presented. The results demonstrate that the fast and slow dynamics can be effectively separated as anticipated, and that the multicontinuum splitting schemes, which exhibit higher efficiency and contrast-independent stability conditions, achieve accuracy comparable to the fully implicit schemes based on multicontinuum homogenization.

We remark that our proposed approach could also be applied to other multiscale problems. For more complicated coefficients, we can introduce as many continua as necessary such that their linear combinations can represent dynamics at different speeds. However, the optimal identification of continua for complex problems remains an open question. Achieving such optimal identification can significantly reduce the degrees of freedom and lower the model order while preserving accuracy, making it an important direction for future research.

\bibliographystyle{unsrt}
\bibliography{ref}

\end{document}